\documentclass[12pt,a4paper,reqno]{amsart}

\usepackage[lmargin=1.4 in, rmargin= 1.4in, tmargin=1.7in]{geometry}

\title{Templicial nerve of an $A_\infty$-category}

\author{Violeta Borges Marques}
\address[Violeta Borges Marques]{Universiteit Antwerpen, Departement Wiskunde, Middelheimcampus, Middelheimlaan 1, 2020 Antwerp, Belgium}
\email{Violeta.BorgesMarques@uantwerpen.be}

\author{Arne Mertens}
\address[Arne Mertens]{Universiteit Antwerpen, Departement Wiskunde, Middelheimcampus, Middelheimlaan 1, 2020 Antwerp, Belgium}
\email{arne.mertens@uantwerpen.be}

\thanks{
This project has received funding from the European Research Council (ERC) under the European Union’s Horizon 2020 research and innovation programme (grant agreement No. 817762).
}

\makeatletter
\@namedef{subjclassname@2020}{
  \textup{2022} Mathematics Subject Classification}
\makeatother



\usepackage{graphicx}
\usepackage[english]{babel}
\usepackage{amsmath}
\usepackage{amssymb}
\usepackage{amsthm}
\usepackage[all,cmtip]{xy}
\usepackage{cancel}
\usepackage[alpine]{ifsym}
\usepackage{stmaryrd}
\usepackage{comment}
\usepackage{tikz}
\usepackage{tikz-cd}
\usetikzlibrary{matrix}
\usepackage[toc,page]{appendix}
\usepackage{hyperref}
\usepackage{mathtools}
\usepackage{float}
\usepackage{hhline}
\usepackage{comment}
\usepackage[shortlabels]{enumitem}
\usepackage[style=alphabetic, maxnames = 50, maxalphanames=50, urldate=long, firstinits=true]{biblatex}
\DeclareMathOperator{\cvee}{\vee\cdots\vee}

\DeclareMathOperator{\Ob}{Ob}

\DeclareMathOperator{\Hom}{Hom}
\DeclareMathOperator{\id}{id}

\DeclareMathOperator{\Fun}{Fun}

\DeclareMathOperator{\Set}{Set}

\DeclareMathOperator{\Mod}{Mod}
\DeclareMathOperator{\Vect}{Vect}

\DeclareMathOperator{\Ch}{Ch}

\DeclareMathOperator{\SSet}{SSet}

\DeclareMathOperator{\Quiv}{Quiv}
\DeclareMathOperator{\Cat}{Cat}

\DeclareMathOperator{\Gr}{Gr}

\newcommand{\fint}{\mathbf{\Delta}_{f}} 
\newcommand{\simp}{\mathbf{\Delta}} 
\newcommand{\nec}{\mathcal{N}ec} 

\newcommand{\ts}{S_{\otimes}} 

\newcommand{\suAinftyCat}{A_\infty\Cat_\text{su}}

\newcommand{\lra}{\longrightarrow}

\newcommand{\K}{\mathbb{K}}

\newcommand{\vvv}{\mathcal{V}}

\newcommand{\aaa}{\mathcal{A}}
\newcommand{\bbb}{\mathcal{B}}

\newcommand{\ccc}{\mathcal{C}}

\newcommand{\op}{\mathrm{op}} 

\newcommand{\necop}{\mathcal{N}ec^{\op}}

\DeclareMathOperator{\ac}{ac}
\DeclareMathOperator{\ine}{in}

\newtheorem{Thm}{Theorem}[section]
\newtheorem*{Thm*}{Theorem}
\newtheorem{Lem}[Thm]{Lemma}
\newtheorem{Prop}[Thm]{Proposition}
\newtheorem{Cor}[Thm]{Corollary}

\theoremstyle{definition}
\newtheorem{Def}[Thm]{Definition}
\newtheorem{Ex}[Thm]{Example}
\newtheorem{Exs}[Thm]{Examples}
\newtheorem{Con}[Thm]{Construction}
\newtheorem{Not}[Thm]{Notation}

\theoremstyle{remark}
\newtheorem{Rem}[Thm]{Remark}

\begin{document}

\begin{abstract}
The framework of templicial vector spaces was put forth in \cite{lowen2024enriched} as a suitable generalization of simplicial sets in order to develop a theory of enriched quasi-categories, called quasi-categories in vector spaces. We construct a lift of Faonte's $A_{\infty}$-nerve \cite{faonte2017simplicial} which lands in templicial vector spaces. Further, we show that when restricted to dg-categories, this nerve recovers the templicial dg-nerve of \cite{lowen2023frobenius}, and that the nerve of any $A_{\infty}$-category is a quasi-category in vector spaces.
\end{abstract}

\maketitle

\tableofcontents

\section{Introduction}
\par $A_\infty$-algebras and $A_\infty$-spaces were introduced by Stasheff in 1963 in his study of $H$-spaces \cite{Stasheff1963Homotopy}. They are algebraic structures suitable for encoding the homotopy associative concatenation product on the loop space $\Omega_x (X)$ of a based topological space $(X,x)$. In the 90's Fukaya introduced $A_\infty$-categories, many objects versions of $A_\infty$-algebras, in the context of symplectic topology \cite{fukaya1993morse}. Given a symplectic manifold $(M^{2n}, \omega)$, Fukaya constructed an $A_\infty$-category counting the intersections of Lagrangian submanifolds, $Fuk(M,\omega)$. These objects received renewed attention after Konstevich's speech at the 1994 ICM, where he gave the modern formulation of the Homological Mirror Symmetry Conjecture which states that for a given Calabi-Yau (CY) manifold $Y$, there exists a mirror CY manifold $\hat{Y}$ such that $D(Fuk(Y))\simeq D(\text{Coh}(\hat{Y}))$, where $D$ denotes the derived category and $\text{Coh}$ denotes the dg-category of coherent sheaves, \cite{Kontsevich1997Homological}. The search for mirror pairs has motivated the construction of a plethora of $A_\infty$-algebras inspired by Mathematical Physics, Geometry and Representation Theory. Nowadays $A_\infty$-algebras are better conceptually understood in the general framework of Operad Theory, where the $A_\infty$-operad is calculated as the minimal cofibrant resolution of the associative operad $\text{Assoc}$ \cite{Loday2012algebraic}. For treatments of diverse aspects of the theory and applications of $A_\infty$-categories see \cite{keller2001introduction}\cite{lefevrehasegawa}
\par In \cite{faonte2017simplicial}, Faonte established a connection between $A_\infty$-categories and the highly succesful quasi-categories of Boardman-Vogt \cite{boardmanvogt}, Joyal \cite{joyal2002quasi} and Lurie \cite{lurie2009higher}. The author constructed a functor called the simplicial nerve of a (strongly unital) $A_\infty$-category, $N^{A_\infty}: \suAinftyCat \to \SSet$, and showed that $N^{A_\infty}(\mathcal{A})$ is a quasi-category for every strictly unital $A_\infty$-category $\mathcal{A}$. This construction strictly generalizes Lurie's dg-nerve $N^{dg}: dg\Cat\to \SSet$ \cite{lurie2016higher}.
\par Templicial objects in a monoidal category were recently introduced in \cite{lowen2023frobenius}, \cite{lowen2024enriched} as an appropriate generalization of simplicial sets upon replacing $\Set$ by a not necessarily cartesian monoidal category $\vvv$. Within this framework, the authors go on to define quasi-categories in $\vvv$ as a counterpart of quasi-categories “in sets”. Although more involved in terms of algebraic structure, these templicial objects - comprised of $\vvv$-objects - are just as tangible as simplicial sets. Of particular interest is the case of templicial modules, i.e., $\vvv=\Mod \K$ with $\K$ a unital commutative ring. In this context, a templicial lift of Lurie's dg-nerve was constructed, allowing one to regard quasi-categories in modules as a kind of “weak dg-categories concentrated in homologically positive degree". More precisely, there is a canonical "underlying simplicial set" functor $\Tilde{U}: \ts\Mod \K \to \SSet$ and the templicial dg-nerve $N_\K^{dg}: dg\Cat \to \ts\Mod k$ satisfies $\Tilde{U}\circ N^{dg}_\K\cong N^{dg}$.
Another great feature of templicial modules and quasi-categories in modules is that they are amenable to algebraic deformation theory; the study of infinitesimal deformations of templicial modules was initiated in \cite{BMLM23}.
\par The main goal of this paper is to construct a templicial lift of Faonte's $A_\infty$-nerve, $N^{A_\infty}_\K: \suAinftyCat \to \ts\Mod \K$, where $\K$ is now assumed to be a field. In Section \ref{Sec:AooNerve}, we define the templicial $A_\infty$-nerve and show that it is well defined and functorial; this involves extensive verifications that are done directly on $N^{A_\infty}_\K(\aaa)$ for a strongly unital $A_\infty$-category $\aaa$. As the simplicial $A_\infty$-nerve is induced by a cosimplicial dg-category (see Subsection \ref{Subsubsec:SimpAooNerve}), we expect the templicial $A_\infty$-nerve to fit the recent framework of templicial nerves induced by conecklicial objects of \cite{mertens2023nerves}. This would greatly simplify verifications and is currently under investigation. In Section \ref{Sec:MainResults}, we show the three main features of our construction:
\begin{enumerate}
    \item Corollary \ref{Cor:UnderlyingSSet}: $N^{A_\infty}_\K$ is a templicial lift of Faonte's simplicial $A_\infty$-nerve, in the sense that $\Tilde{U}\circ N^{A_\infty}_\K\cong N^{A_\infty}$;
    \item Corollary \ref{Cor:dgNerve}: the templicial $A_\infty$-nerve strictly generalizes the templicial dg-nerve, in the sense that for a dg-category $\mathcal{C}$, regarded as a strictly unital $A_\infty$-category, $N^{A_\infty}_\K(\mathcal{C})\cong N^{dg}_\K(\mathcal{C})$;
    \item Theorem \ref{Thm:QCat}: for any $A_\infty$-category $\mathcal{A}$, $N^{A_\infty}_\K(\mathcal{A})$ is a quasi-category in vector spaces.
\end{enumerate}

\subsection{Signs and conventions}
\par Throughout this article we will make use of homological degree, i.e., differentials \emph{decrease} degree. This is contrary to \cite{faonte2017simplicial}, but has no impact on signs. 
\par Another difference from \cite{faonte2017simplicial} is the order of the inputs of the multiplications and of the higher components of the $A_\infty$-functors: we will use \emph{multiplicative} notation while in \cite{faonte2017simplicial} \emph{compositional} notation is used. More explicitly, for an $A_\infty$-category $\aaa$ and $x_0,...,x_n\in \Ob \aaa$, the author in loc. cit. expresses the multiplications as
\begin{equation*}
    \tilde{m}_n: \aaa_{\bullet}(x_{n-1},x_n)\otimes \cdots \otimes \aaa_{\bullet}(x_0,x_1)\to \aaa_{\bullet}(x_0,x_1)
\end{equation*}
\noindent while we shall express them as 
\begin{equation*}
    m_n = \tilde{m}_{n}\tau_{n}: \aaa_{\bullet}(x_0,x_1)\otimes \cdots \otimes \aaa_{\bullet}(x_{n-1},x_n)\to \aaa_{\bullet}(x_0,x_n)
\end{equation*}
For any $C_{1},\dots,C_{n}\in \Ch\K$, $\tau_{n}$ denotes the symmetry of the monoidal category $(\Ch \K, \otimes , \K{[0]})$, which introduces an extra sign:
\[
\begin{tikzcd}[row sep=0]
    C_1\otimes \cdots \otimes C_n \arrow{r}{\tau_n}& C_n \otimes \cdots \otimes C_1\\
    x_1\otimes \cdots \otimes x_n \arrow[mapsto]{r}& (-1)^{\sigma(|x_1|,...,|x_n|)}x_n\otimes \cdots \otimes x_1
\end{tikzcd}
\]
with $\sigma(i_1,...,i_n)=\sum_{1\leq a < b \leq n} i_a i_b$.

We shall make use of the sign convention of \cite{faonte2017simplicial}\cite{lefevrehasegawa}, with signs appropriately modified in order to accommodate the symmetry $\tau_{n}$. Note that the resulting signs coincide with those of \cite{keller2001introduction}\cite{getzler1990Ainfty}.

\vspace{0,3cm}
\noindent \emph{Acknowledgement.}
The authors would like to thank Wendy Lowen for her helpful discussions and support during the writing of this paper.

\section{Preliminaries}
\label{Sec:Pre}
\subsection{Strictly unital $A_\infty$-categories and their simplicial nerve}
\label{Subsec:AooCats}
\par In this section we collect some basic definitions and facts about $A_\infty$-categories. For an extensive treatment, we refer to \cite{lefevrehasegawa} and \cite{faonte2017simplicial}. Throughout the paper, let $\K$ be a field; all vector spaces are $\K$-vector spaces.
\begin{Def}
\label{Def:AooCat}
A (small) \emph{$A_\infty$-category} $\aaa$ is the data of 
\begin{enumerate}[(a)]
    \item a set of objects $\Ob(\aaa)$;
    \item for any pair of objects $x,y\in\Ob(\aaa)$, a graded vector space $\aaa_{\bullet}(x,y)$;
    \item for any $k\geq 1$ and any sequence of objects $x_0,...,x_k$, a morphism of degree $k-2$, called the \emph{multiplication}:
    \begin{equation*}
        m_k: \aaa_{\bullet}(x_0,x_1)\otimes \cdots\otimes \aaa_{\bullet}(x_{k-1}, x_k) \to \aaa_{\bullet}(x_0,x_k)
    \end{equation*}

\end{enumerate}
\noindent such that for any $k\geq 1$ the following equation holds
\begin{equation}
\label{Eq:AooRelation}
\sum_{\substack{r+s+t=k\\ r,s\geq 0, s > 0}}(-1)^{r+st} m_{r+1+t}(\id^{\otimes r}\otimes m_s \otimes \id^{\otimes t})=0
\end{equation}
\noindent We will often write $m_1=\partial $ and $m_2=m$.
\par An $A_\infty$-category $\aaa$ is called \emph{strictly unital} if there is the data of
\begin{enumerate}[(a),start=4]
    \item for every object $x \in \Ob (\aaa)$, an element $1_x\in Z_0 \aaa (x,x)$, called the \emph{identity at $x$}.
\end{enumerate}
\noindent such that for any $x\in\Ob(\aaa)$, $k>2$ and $1\leq j \leq k$
\begin{equation}
\label{eq:AooUnit}
    m(\id \otimes u_x) = m(u_x\otimes\id)=\id\quad \text{and}\quad m_k(\id^{j-1}\otimes u_x\otimes \id^{k-j})=0
\end{equation}

\noindent where we identified $1_x$ with a morphism $u_x: \K[0]\to \aaa_{\bullet}(x,x)$.
\end{Def}

It is worth to analyze equation \eqref{Eq:AooRelation} for low values of $k$. 
\begin{align*}
\label{Eq:AooRelk=1,2,3}
(k=1)&\quad \partial\circ \partial = 0\\
(k=2)&\quad \partial m = m(\id \otimes \partial) + m(\partial \otimes \id )\\
(k=3)&\quad \partial m_3 + m_3(\partial \otimes \id^{\otimes 2}) + m_3(\id \otimes \partial \otimes \id) + m_3( \id^{\otimes 2}\otimes\partial) \\
&\quad = m(\id \otimes m) - m(m \otimes \id )
\end{align*}
\par The first equation means that $\aaa_{\bullet}(x,y)$ are chain complexes for all $x,y\in \Ob(\aaa)$ when equipped with $m_1$, the second expresses the fact that the composition $m$ is a map of chain complexes and the third spells out that $m$ is associative up to the homotopy $m_3$.
\begin{Ex}
\label{Ex:dgCats} 
Let $\mathcal{C}_{\bullet}$ be a dg-category with differential $\partial$ and composition $m$. Then $\mathcal{C}_{\bullet}$ is naturally a strictly unital $A_\infty$-category with the same objects, the same chain complexes, the same identities, $m_1 = \partial$, $m_2 = m$ and $m_k = 0$ for $k > 2$.
\end{Ex}
\begin{Ex}\label{Ex:AooCatPullback}
\par Let $\aaa$ be a strictly unital $A_\infty$-category and $f: O\to \Ob \aaa$ be a map of sets. Then $f^*(\aaa)$ is a stritcly unital $A_\infty$-category with object set $T$, $f^{*}\aaa_{\bullet}(x,y)=\aaa_{\bullet}(f(x), f(y))$, units $u_{f(x)}$ for all $x,y\in O$ and multiplications
$$
m_k: \aaa_{\bullet}(f(x_0),f(x_1))\otimes \cdots\otimes \aaa_{\bullet}(f(x_{k-1}), f(x_k)) \to \aaa_{\bullet}(f(x_0),f(x_k))
$$
given by those of $\mathcal{A}$ for all $x_{0},\dots,x_{k}\in O$.
\end{Ex}
\begin{Def}
\label{Def:AooMor}
\par Let $\aaa$ and $\bbb$ be two $A_\infty$-categories. An \emph{$A_\infty$-functor} $f: \aaa \to \bbb$ is the data of
\begin{enumerate}[(a)]
    \item a map of sets $f_0: \Ob(\aaa)\to \Ob(\bbb)$;
    \item for $k\geq 1$ and any sequence of objects $x_0,...,x_k\in\Ob(\aaa)$, a morphism of degree $k-1$:
    \begin{equation*}
        f_k: \aaa_{\bullet}(x_0,x_1)\otimes \cdots \otimes \aaa_{\bullet}(x_{k-1}, x_k) \to \bbb_{\bullet}(f(x_0), f(x_k))
    \end{equation*}
\end{enumerate}
\noindent such that for any $k\geq 1$ the following equation holds
\begin{equation}
\label{Eq:AooMorRel}
\sum_{\substack{r+s+t=k\\ r,t\geq 0, s > 0}}\!\!\!\!(-1)^{r+st}f_{r+1+t}\left(\id^{\otimes r}\otimes m^\aaa_s \otimes \id^{\otimes t}\right)\,=\!\!\!\sum_{\substack{i_1+\cdots + i_r=k\\ r > 0, i_{j} > 0}} \!\!\!\!(-1)^{{\epsilon_g(i_1,...,i_r)}}m^\bbb_r\left(f_{i_1}\otimes \cdots \otimes f_{i_r}\right)
\end{equation}
\noindent where $\epsilon_g(i_1,...,i_r)=\sum_{l=1}^{r-1}(r-l)(i_{l}-1)$. If moreover $\aaa$ and $\bbb$ are strictly unital, we require that for any $x\in \Ob(\aaa)$, $k>1$ and $1\leq j \leq k $
\begin{equation}
\label{eq:AooFunctorUnit}
    f_1 u_x=u_{f_0(x)}\quad \text{and}\quad f_k(\id^{j-1}\otimes u_x\otimes \id^{k-j})=0
\end{equation}
\end{Def}
\par For low values of $k$, equation \eqref{Eq:AooMorRel} becomes 
\begin{align*}
\label{Eq:AooRelk=1,2,3}
(k=1)&\quad \partial^\bbb f_1 = f_1\partial^\aaa \\
(k=2)&\quad \partial^\bbb f_2 + f_2(\partial^\aaa \otimes \id)+ f_2(\id \otimes \partial^\aaa ) = f_1 m^\aaa - m^\bbb(f_1 \otimes f_1)
\end{align*}
\par The first equation states that $f_1$ is a map of chain complexes and the second equation that $f_1$ commutes with the composition up to a homotopy $f_2$.

\begin{Con}
\par Let $f: \aaa \to \bbb$ and $g: \bbb \to \ccc$ be two $A_\infty$-functors between (strictly unital) $A_\infty$-categories. Then the composition $g\circ f: \aaa \to \ccc$ is given by 
\begin{align*}
(g\circ f)_0 &= g_0\circ f_0\\
(g\circ f)_k &= \sum_{\substack{i_1+\cdots + i_r = k\\ r > 0, i_{j} > 0}} (-1)^{\epsilon_g(i_1,...,i_r)}g_r\left(f_{i_1}\otimes \cdots \otimes f_{i_r}\right)
\end{align*}
\noindent for $k\geq 1$.
\end{Con}
\begin{Not}
\label{Not:CatOfAooCats}
\par We will denote the category of strictly unital $A_\infty$-categories and $A_\infty$-functors $\suAinftyCat$ and from now on we will refer to objects of $\suAinftyCat$ simply as $A_\infty$-categories. 
\end{Not}

\subsubsection{The simplicial $A_\infty$-nerve}
\label{Subsubsec:SimpAooNerve}
\par In this subsection we recall Faonte's simplicial $A_\infty$-nerve. This nerve is generated by a cosimplicial $A_\infty$-category 
$$
A_\infty{[\Delta^{(-)}]}: \Delta \to dg\Cat\to \suAinftyCat 
$$
of which we will describe the values on objects. For any $n\geq 0$, $A_\infty{[\Delta^n]}$ is the dg-category with objects $\{0,...,n\}$ and hom-chain complexes concentrated in degree $0$:
$$
A_\infty{[\Delta^{n}]}(i,j)=\begin{cases}
    \K\cdot (i,j) & i \leq j\\
    0 & i > j\\
\end{cases}
$$
The multiplication in $A_{\infty}[\Delta^{n}]$ is determined by
$$
m((i,j)\otimes (j,k))=(i,k)\quad\text{ for all }0\leq i\leq j \leq k\leq n
$$
\begin{Def}
\label{Def:SimplAooNerve}
The \emph{simplicial $A_\infty$-nerve} is the functor 
\[
\begin{tikzcd}[row sep=0]
    \suAinftyCat \arrow{r}{N^{A_\infty}}& \SSet\\
    \aaa \arrow[mapsto]{r}{}& \left({[n]}\mapsto \Hom_{\suAinftyCat}(A_\infty[\Delta^n], \aaa)\right)
\end{tikzcd}
\]
\end{Def}
\par Let $\aaa$ be an $A_\infty$-category and $n\geq 0$. An element in $N^{A_\infty}(\aaa)_n$ is given by the following data
\begin{enumerate}[(a)]
    \item $n+1$ objects $x_0,...,x_n\in \Ob(\aaa)$;
    \item for any sequence $0\leq i_0 < \cdots < i_k \leq n$, an element
    \[
    a_{i_0,...,i_k}\in \aaa_{k-1}(x_{i_0}, x_{i_k})
    \]
\end{enumerate}
\noindent such that 
\begin{align}
\begin{split}
\label{Eq:AooNerve}
\partial(a_{i_0,...,i_k})=\sum_{j=1}^{k-1}&(-1)^{j-1} a_{i_0,...,\not i_j,...,i_k} -\\&- \!\!\!\sum_{\substack{2\leq r \leq j\\0 < j_1 < \cdots < j_{r-1} < k}}\!\!\!\!\!\!\!\!\! (-1)^{\epsilon_c(j_1,...,j_{r-1})} m_r(a_{i_0,...,i_{j_1}}\otimes \cdots \otimes a_{i_{j_{r-1}},..., i_k})
\end{split}
\end{align}
where $\epsilon_c(j_{1},\dots ,j_{l-1}) = \sum_{j=2}^{l}(j-1)(k_{j} - k_{j-1}-1) = k_{1} + \dots + k_{l-1} -  \frac{(l-1)l}{2}$. We refer to \cite{faonte2017simplicial} for explicit descriptions of the actions of the simplicial maps and record here two relevant properties of the simplicial $A_\infty$-nerve.

\begin{Thm}[Cor 2.10, \cite{faonte2017simplicial}]
Let $\aaa$ be an $A_\infty$-category. Then $N^{A_\infty}(\aaa)$ is a quasi-category.
\end{Thm}
\begin{Thm}[Prop 2.15, \cite{faonte2017simplicial}]
Let $\mathcal{C}_{\bullet}$ be a dg-category considered as an $A_\infty$-category as in Example \ref{Ex:dgCats}. Then $N^{A_\infty}(\mathcal{C})\cong N^{dg}(\mathcal{C})$.
\end{Thm}

\subsection{The category of necklaces}
\par Let $\SSet_{*,*} = \SSet_{\partial\Delta^{1}/}$ denote the category of bipointed simplicial sets. The category $\nec$ is the full subcategory of $\SSet_{*,*}$ spanned by all \emph{necklaces}, that is sequences $\Delta^{n_{1}}\vee ...\vee \Delta^{n_{k}}$ of standard simplices, called \emph{beads}, which are glued at their endpoints (where $\vee$ denotes the wedge sum). Necklaces first appeared in \cite{baues1980geometry} (under a different name) and were later employed and popularised in \cite{dugger2011rigidification}.
\par Let $\simp$ be the simplex category, which consists of the finite ordinals $[n] = \{0,...,n\}$ for $n\geq 0$ as objects with order preserving maps as morphisms. We will make use of the finite interval category $\fint$, which is the subcategory of $\simp$ with the same objects and containing those morphisms $f: [m]\rightarrow [n]$ that preserve the endpoints, that is, $f(0) = 0$ and $f(m) = n$. The category $\fint$ is strict monoidal for $[n] + [m] =[n+m]$. It was shown independently in \cite[Proposition 3.4.2]{grady2023extended} and \cite[Proposition 3.4]{lowen2023frobenius} that $\nec$ admits the following combinatorial description:
\begin{itemize}
\item The objects of $\nec$ are all pairs $(T,p)$ with $p\geq 0$ an integer and $T\subseteq [p]$ a subset containing $\{0 < p\}$.
\item A morphism $(T,p)\rightarrow (U,q)$ in $\nec$ is a morphism $f: [p]\rightarrow [q]$ in $\fint$ such that $U\subseteq f(T)$.
\end{itemize}
Here, a subset $T = \{0 = t_{0} < t_{1} < ... < t_{k} = p\}$ of $[p]$ corresponds to the necklace $\Delta^{t_{1}}\vee \Delta^{t_{2}-t_{1}}\vee ...\vee \Delta^{p-t_{k-1}}$. For example, $T = \{0 < p\}$ is the single simplex $\Delta^{p}$ while $T = [p]$ is a sequence of edges. Then the wedge sum $\vee$ is given as follows:
$$
(T,p)\vee (U,q) = (T\cup (p + U), p+q)
$$
for any necklaces $(T,p)$ and $(U,q)$. This makes $\nec$ into a monoidal category with monoidal unit given by $(\{0\},0)$.
\begin{Def}
Let $T = \{0 = t_{0} < t_{1} < \dots < t_{k} = p\}$ be a necklace. In other words, $T = \Delta^{n_{1}}\vee \dots \vee \Delta^{n_{k}}$ with $n_{i} = t_{i} - t_{i-1}$. Then
\begin{itemize}
    \item the \emph{bead length} of $T$ is $\ell(T) = k$;
    \item the \emph{spine lenght} of $T$ is $\Vert T\Vert = p = n_{1} + \dots + n_{l}$;
    \item the \emph{dimension} of $T$ is $\dim(T) = \Vert T\Vert - \ell(T)$;
    \item the \emph{signature} of $T$ is $\epsilon(T) = \epsilon_{g}(n_{1},\dots,n_{k}) = \epsilon_{c}(t_{1},...,t_{k-1})$
\end{itemize}
\end{Def}

It will be useful to consider some particular classes of necklace maps.
We call a necklace map $f: (T,p)\rightarrow (U,q)$ \emph{inert} if the underlying morphism $f: [p]\rightarrow [q]$ in $\fint$ is the identity, and we call $f$ \emph{active} if $f(T) = U$. Every necklace map can be uniquely decomposed as an active map followed by an inert map, as shown bellow
\begin{equation*}
    (T,p) \xrightarrow{f^{\ac}} (f(T),q)\xrightarrow{f^{\ine}} (U,q) 
\end{equation*}
Further, we call $f$ \emph{injective} (resp. \emph{surjective}) if the underlying morphism $f: [p]\rightarrow [q]$ in $\fint$ is injective (resp. surjective). 

It was shown in \cite[Lemmas 3.5 and 3.6]{BMM2024necklaces} that a necklace map is inert if and only if it is the wedge $\nu_i\vee\cdots\vee \nu_k$ of the inclusions
\[
\nu_i: \Delta^{n_i^1}\vee\cdots\vee\Delta^{n_i^{l_i}}\hookrightarrow\Delta^{n_i^1+\cdots+{n_i^{l_i}}}
\]
and a necklace map is active if and only if it is the wedge $f_1\vee\cdots\vee f_k$ of necklace maps $f_i: \Delta^{n_i}\to \Delta^{m_i}$ induced by morphisms $[n_{i}]\rightarrow [m_{i}]$ in $\fint$.

\begin{Prop}
\label{Prop:Factorization}[\cite{BMM2024necklaces}]
The (epi,mono) factorization system on $\SSet_{*,*}$ restricts to $\nec$. The epimorphisms are precisely the active surjective necklace maps and the monomorphisms are precisely the injective necklace maps.
\end{Prop}
\par We denote the wide subcategory of $\nec$ spanned by the injective (resp. active surjective) maps by $\nec_{+}$ (resp. $\nec_{-}$).
\begin{Def}
\label{Def:SpineCollapsing}
Let $f = f_{1}\cvee f_k$ be an active surjective necklace map with $f_i: \Delta^{n_{i}}\to \Delta^{m_{i}}$. The map $f$ is \emph{bead reducing} if $m_i\geq 1$ for all $1\leq i \leq k$ and \emph{spine collapsing} if for all $1\leq i \leq k$, $f_i=\id_{\Delta^{n_{i}}}$ or $f_i: \Delta^{1}\rightarrow \Delta^{0}$.
\end{Def}

\begin{Lem}[\cite{BMM2024necklaces}]
\label{Lem:SpineCollandDim}
\begin{enumerate}[1.]
    \item Let $f: X \to Y$ be an active surjective necklace map. Then $\dim(X)\geq \dim(Y)$ and there is equality if and only if $f$ is spine collapsing.
    \item The injective maps in $\nec$ are direct divisible with respect to $\vee$. That is, for any injective map $g: U\hookrightarrow T_{1}\vee T_{2}$, there exists unique injective maps $g_{i}: U_{i}\hookrightarrow T_{i}$ for $i\in \{1,2\}$ such that $g = g_{1}\vee g_{2}$.
\end{enumerate}

\end{Lem}

Finally, let us introduce some new notation for the purposes of this paper.

\begin{Not}
Let $U = \Delta^{n_{1}}\vee \dots \vee \Delta^{n_{l}}$ be a necklace and $1\leq k\leq l$. Then we denote
$$
U^{< k} = \Delta^{n_{1}}\vee \dots \vee \Delta^{n_{k-1}}\quad \text{and}\quad U^{> k} = \Delta^{n_{k+1}}\vee \dots \vee \Delta^{n_{l}}
$$
Now given necklace maps $g: U\to S$ and $f: T\to \Delta^{n_k}$, we write
$$
g \circ_k f = g \circ (\id_{U^{<k}} \vee f \vee \id_{U^{> k}}): U^{< k}\vee T \vee U^{> k} \to S
$$
\end{Not}

\begin{Def}
Let $S=\Delta^{n_1}\vee \cdots \vee \Delta^{n_l}\in \nec$, $\mu: T \hookrightarrow S$ be an inert map. Note that there exist unique inert necklace maps $\mu_{i}: T_{i}\hookrightarrow \Delta^{n_{i}}$ for all $1\leq i\leq l$ such that $\mu = \mu_{1}\vee \dots \vee \mu_{l}$. The \emph{signature} of $\mu$ is defined as
\begin{equation*}
    \varphi(\mu)=\sum_{1\leq i < j \leq l}\dim(T_i)\left(\ell(T_j)-1\right)
\end{equation*}

Given $1\leq k\leq l$, then in particular when $\mu_{i} = \id_{\Delta^{n_{i}}}$ for $i\neq k$ and $\mu_{k} = \nu: T\hookrightarrow \Delta^{n_{k}}$, we write
$$
\varphi_{k}(\nu) = \varphi(\id^{\vee k-1}\vee \nu \vee \id^{\vee l-k}) = \dim(S^{<k})(\ell(T)-1)
$$
\end{Def}

\begin{Lem}
\label{Lem:MagicFormula}
 Let $S=\Delta^{n_1}\vee\cdots \vee  \Delta^{n_l}\in \nec$, $\mu: T \hookrightarrow S$ be an inert map and $T=T_1\vee \cdots \vee T_l$ the decomposition induced by $\mu$. Then 
 \begin{align*}
     \varphi(\mu)=&\epsilon(T) + \epsilon(S) -\epsilon(T_1)-\cdots - \epsilon(T_l) -\epsilon_{g}(\ell(T_1),...,\ell(T_2))
 \end{align*}
 In particular, for an inert map $\nu: W\hookrightarrow \Delta^{n_{k}}$ with $1\leq k\leq l$, we have
 $$
 \epsilon(V^{< k}\vee W \vee V^{> k}) = \epsilon(V) + \epsilon(W) + \varphi_{k}(\nu) -\left(\ell(W)-1\right)\ell(V^{> k})
 $$
\end{Lem}
\begin{proof}
    This is a direct verification.
\end{proof}

\subsection{Templicial objects in a monoidal category}
\par Templicial objects (see \S \ref{Subsec:temp}) in a suitable monoidal category $\vvv$ were introduced in \cite{lowen2024enriched} as an appropriate generaliszation of simplicial sets upon replacing $\Set$ by a not necessarily cartesian monoidal category $\vvv$. In this section, we will focus on the case of templicial vector spaces, that is, letting $\vvv$ be the category of vector spaces over the field $\K$. The case $\vvv = \Set$ recovers simplicial sets as templicial objects in $\Set$.
In \S \ref{parqcat}, we recall quasi-categories in vector spaces. These are the templicial vector spaces which satisfy a familiar horn lifting property. In order to express this property, one requires necklicial vector spaces (see \S \ref{parnec}). 

\subsubsection{Templicial vector spaces}\label{Subsec:temp}
Recall that a $\K$-linear category $\ccc$ with object set $O$ consists of a collection $(\ccc(a,b))_{a,b\in O}$ with $\ccc(a,b)\in \Vect(\K)$ endowed with multiplications
$$
\mathcal{C}(a,b)\otimes_\K \mathcal{C}(b,c)\longrightarrow \mathcal{C}(a,c)
$$
and units $\K \to \ccc(a,a)$ satisfying the usual associativity and unitality axioms.

On the other hand, given a set $O$, we refer to a collection $Q = (Q(a,b))_{a,b\in O}$ with $Q(a,b)\in \Vect(\K)$ as a \emph{$\K$-linear quiver} with $O$ its set of \emph{vertices}. A \emph{quiver morphism} $f: Q\rightarrow P$ is a collection $(f_{a,b})_{a,b\in O}$ of $\K$-linear maps $f_{a,b}: Q(a,b)\rightarrow P(a,b)$. We thus obtain a category
$$
\K\Quiv_{O}
$$
of $\K$-linear quivers with a fixed vertex set $O$. This category can be equipped with a monoidal structure $(\otimes_{O},\K_{O})$ such that a monoid in $\K\Quiv_{O}$ is precisely to a $\K$-linear category with object set $O$. We refer to \cite{lowen2024enriched} for more details. 
\begin{Def}\label{definition: temp. obj.}
A \emph{templicial vector space} is a pair $(X,O)$ with $O$ a set and
$$
X: \fint^{op}\longrightarrow \K\Quiv_{O}
$$
a strongly unital, colax monoidal functor. 
\end{Def}

\par We refer to loc. cit. for the definition of the category $\ts\Vect(\K)$ of templicial vector spaces with varying object sets.

\par There is a free-forgetful adjunction $\tilde{F}: \Set \rightleftarrows \Vect(\K): \tilde{U}$ induces an adjunction
\begin{equation}\label{eqtildeF}
	\tilde{F}_\K: \SSet \rightleftarrows \ts \Vect(\K): \Tilde{U}.
\end{equation}
We refer to $\Tilde{U}(X)$ for $X\in\ts\Vect(\K)$ as the \emph{underlying simplicial set} of $X$.

Let $(X,O)$ be a templicial vector space. The structure of $X$ as a functor $\fint^{op}\rightarrow \K\Quiv_{O}$ is equivalent to a collection of $\K$-quivers $(X_{n})_{n\geq 0}$ as well as $\K$-linear quiver morphisms
\begin{align*}
d_{j} 
: X_{n}\longrightarrow X_{n-1}\quad \text{for all integers }0 < j < n\\
s_{i} 
: X_{n}\longrightarrow X_{n+1}\quad \text{for all integers }0\leq i\leq n
\end{align*}
called the \emph{inner face morphisms} and \emph{degeneracy morphisms} respectively, which satisfy the usual simplicial identities. The colax monoidal structure of $X$ provides it with \emph{comultiplications}, which are quiver morphisms
$$
\mu_{k,l}: X_{k+l}\longrightarrow X_{k}\otimes_{O} X_{l}
$$
for all $k,l\geq 0$ satisfying coassociativity and naturality conditions. Further, $X$ has a \emph{counit} $\epsilon: X_{0}\xlongrightarrow{\cong} \K_{O}$, which is assumed to be an isomorphism by the strong unitality. Note that, contrary to a simplicial object, $X$ does not have any outer face morphisms $d_{0}, d_{n}: X_{n}\longrightarrow X_{n-1}$. Instead, $X$ is equipped with the comultiplications $\mu_{k,l}$ which serve as a replacement for the outer faces in the (non-cartesian) monoidal context.

We end this section with the main motivating example for using templicial vector spaces in order to model higher enriched categories.

\begin{Ex}\cite[\S 2.2]{lowen2024enriched}\label{exnerve}
Let $\ccc$ be a $\K$-linear category with object set $O$. The \emph{templicial nerve} of $\ccc$ is the templicial vector space $N_{\K}(\ccc)$  with
$$
N_{\K}(\ccc)_n = \ccc^{\otimes_O n},
$$
with inner face maps induced by the multipliations of $\ccc$, degeneracies induced by the units of $\ccc$, and comultiplications given by the canonical isomorphisms $\ccc^{\otimes_S k+l} \cong \ccc^{\otimes_S k} \otimes_O \ccc^{\otimes_S l}$.
\end{Ex}

\subsubsection{Necklicial objects}\label{parnec}
 \par Necklace maps allow us to treat the inner face maps and degeneracy maps of a templicial vector space $(X,O)$ on the one hand, and its comultiplications on the other hand, on the same footing. More concretely, given any necklace $T = \{0 = t_{0} < t_{1} < ... < t_{k} = p\}$, we set
$$
X_{T} = X_{t_{1}}\otimes_{O} X_{t_{2}-t_{1}}\otimes_{O} ...\otimes_{O} X_{p-t_{k-1}}\in \K\Quiv_{O}.
$$
This definition can be extended to a functor $X_{\bullet}: \nec^{op}\rightarrow \K\Quiv_{O}$ in which the active necklace maps parameterise the inner face maps and degeneracy maps of $X$ and the inert necklace maps parameterise its comultiplications (see \cite[Construction 3.9]{lowen2024enriched}).

In particular, we can evaluate in any $a,b\in S$ to obtain a functor
\begin{equation}\label{diagram: hom-object of necklace cat. assoc. to temp. obj.}
X_{\bullet}(a,b): \nec^{op}\longrightarrow \Vect(\K).
\end{equation}
In general, a functor $Y: \nec^{op}\rightarrow \Vect(\K)$ will be called a \emph{necklicial vector space}.  The category $\Fun(\necop, \Vect(\K))$ becomes monoidal for the Day convolution, and one may thus consider \emph{necklace categories}, that is categories enriched in $\Fun(\necop, \ \Vect(\K))$. We denote the category of necklace categories by $\K \Cat_{\nec}$. 
\par Given $T,S\in \nec$ and $a,b,c\in O$, the canonical maps 
\begin{equation*}
    X_T(a,b)\otimes_\K X_S(b,c) \to X_{T\vee S}(a,c)
\end{equation*}
\noindent induce compositions $X_\bullet(a,b) \otimes_{Day}X_\bullet(b,c)\to X_\bullet(a,c)$ and the isomorphisms $\K \cong X_0(a,a)=X_{\{0\}}(a,a)$ induce unit morphism $\underline{K} \to X_\bullet(a,a)$. This defines a necklace category $X^{nec}\in \K\Cat_{\nec}$ with object set $O$ and $X^{nec}(a,b)=X_{\bullet}(a,b)$ for all $a,b\in O$.

\begin{Prop}\cite[Theorem 3.12]{lowen2024enriched}
\label{Prop:tempnec}
The construction above defines a fully faithful functor 
\begin{equation*}
\label{eq:nec}
    (-)^{nec}: \ts\Vect(\K) \lra 
	\K \Cat_{\nec}
\end{equation*}
This functor is left adjoint to a functor 
	$(-)^{temp}:
	\K \Cat_{\nec} \lra \ts\Vect(\K)$ and its  essential image is spanned by those $\mathcal{C}\in \K\Cat_{\nec}$ with such that the quiver maps inducing composition
    \begin{equation*}
        \mathcal{C}_T \otimes_{\Ob \mathcal{C}} \mathcal{C}_S \to C_{T\vee S}\in \K\Quiv_{\Ob \mathcal{C}}
    \end{equation*}
    \noindent are isomorphisms.
\end{Prop}

\subsubsection{Quasi-categories in vector spaces}\label{parqcat}
\par Given integers $0\leq j\leq n$, we denote by $\Delta^{n}$ and $\Lambda^{n}_{j}$ the standard $n$-simplex and the $j$th $n$-horn respectively. 

\begin{Def}\label{DefwK}
Let $Y: \nec^{op}\longrightarrow \Vect(\K)$ be a necklicial vector space. We say that $Y$ is \emph{weak Kan} if for all $0 < j < n$ any lifting problem
\[\begin{tikzcd}
	{\tilde{F}(\Lambda^{n}_{j})_{\bullet}(0,n)} & Y \\
	{\tilde{F}(\Delta^{n})_{\bullet}(0,n)}
	\arrow[from=1-1, to=2-1]
	\arrow[from=1-1, to=1-2]
	\arrow[dashed, from=2-1, to=1-2]
\end{tikzcd}\]
where the vertical morphism is induced by the inclusion $\Lambda^{n}_{j}\subseteq \Delta^{n}$, has a solution in $\Fun(\necop, \Vect(\K))$. Here, $\tilde{F}$ is the free functor from \eqref{eqtildeF}. We call a templicial vector space $(X,O)$ a \emph{quasi-category in vector spaces} if the functors $X_{\bullet}(a,b)$ are weak Kan for all $a,b\in O$.
\end{Def}

\begin{Ex}
\label{Ex:QCats}
The following are quasi-categories in vector spaces:
\begin{enumerate}[1.]
	\item \cite[Corollary 5.11]{lowen2024enriched} the nerve $N_{\K}(\ccc)$ of a $\K$-linear category $\ccc$ (see Example \ref{exnerve});
    \item \cite[Corollary 5.12]{lowen2024enriched} the templicial homotopy coherent nerve $N^{hc}_\K(\mathcal{C})$ of a $S\Vect(\K)$-enriched category such that $U(\mathcal{C}(a,b))$ is a Kan complex for every $a,b\in \Ob \mathcal{C}$.
	\item \cite[Corollary 4.15]{lowen2023frobenius} the free templicial vector space $\tilde{F}(X)$, for an ordinary quasi-category $X \in \SSet$.
    \item \cite[Corollary 4.17]{lowen2023frobenius} the templicial dg-nerve $N^{dg}_\K(\mathcal{C})$ of a dg-category $\mathcal{C}_{\bullet}$ over $\K$.
\end{enumerate}
\end{Ex}

\section{The templicial $A_\infty$-nerve}
\label{Sec:AooNerve}
\subsection{The templicial $A_\infty$-nerve of an $A_\infty$-category}
\label{Subsec:AooNerve}
\par Let $(A, m_1=\partial, m_2=m, m_3,...)$ be a (small) $A_\infty$-category over $\mathbb{K}$ concentrated in non-negative degree with object set $\Ob(\mathcal{A})$. We define its \emph{templicial $A_{\infty}$-nerve} $N^{A_{\infty}}_{\mathbb{K}}(\mathcal{A})$ by constructing its associated necklace category $N^{A_{\infty}}_{\mathbb{K}}(\mathcal{A})^{nec}$.

\begin{Def}
Given a necklace $T = \nec$, we define a $k$-quiver
$$
N^{A_{\infty}}_{\mathbb{K}}(\mathcal{A})_{T}\in \mathbb{K}\Quiv_{\Ob \mathcal{A}}
$$
as follows. For any $a,b\in \Ob(\mathcal{A})$, define $N^{A_{\infty}}_{\mathbb{K}}(\mathcal{A})_{T}(a,b)$ as the vector space of collections
\[
y = (y_{g})_{g}\in \bigoplus_{\substack{g: U\hookrightarrow T\\ \text{injective in }\nec}} (s\mathcal{A})_{U}(a,b)
\]
satisfying, for all injective necklace maps $g: U\hookrightarrow T$:
\begin{equation}\label{equation:TAN}
\sum_{k=1}^{l}\sum_{\substack{\nu: S\hookrightarrow \Delta^{n_{k}}\\ \text{inert}}}(-1)^{\epsilon(S)+\varphi_k(\nu)} (\id^{\otimes k-1}\otimes m_{\ell(S)}\otimes \id^{\otimes l-k})(y_{g\circ_k \nu}) = \sum_{j=1}^{d} (-1)^{j-1} y_{g\circ \delta_{i_{j}}}
\end{equation}
in $\overline{T}(s\mathcal{A})_{d-1}(a,b)$ where we've denoted $U = \Delta^{n_{1}}\vee \dots\vee \Delta^{n_{l}}$ and $U^{c} = \{i_{1} < \dots < i_{d}\}$ so that $l = \ell(U)$ and $d = \dim(U)$. Note that in particular, for $T = \Delta^{0}$, we have $N^{A_{\infty}}_{\mathbb{K}}(\mathcal{A})_{0} = \underline{\mathbb{K}}$.

We refer to equation \eqref{equation:TAN} as the \emph{templicial $A_\infty$-nerve (TAN) relation} at $g$.
\end{Def}

\begin{Rem}
In fact, the (TAN) relation at $g$ decomposes into $l$ independent equations, one for each bead of $U$. 
For every $1\leq k\leq \ell(U)$ we have: 
\begin{equation}\label{equation:TANdecomp}
\sum_{\substack{\nu: S\hookrightarrow \Delta^{n_{k}}\\ \text{inert}}}(-1)^{\epsilon(S)+\ell(S)\dim(U^{<k})} (\id^{\otimes k-1}\otimes m_{\ell(S)}\otimes \id^{\otimes l-k})(y_{g\circ_k \nu}) = \sum_{j=1}^{n_k-1}(-1)^{j-1} y_{g\circ_{k} \delta_{j}}
\end{equation}
in $(s\mathcal{A})_{\Delta^{n_{1}}\vee \dots \vee \Delta^{n_{k}-1}\vee \dots \vee\Delta^{n_{l}}}(a,b)$.
\end{Rem}

The assignment $T\mapsto N^{A_{\infty}}_{\mathbb{K}}(\mathcal{A})_{T}$ extends to a functor $\nec^{op}\rightarrow \mathbb{K}\Quiv_{\Ob(\mathcal{A})}$ (see Proposition \ref{proposition:Ainftynervestructuremaps}). In order to define this functor, we first need to consider an auxiliary strong monoidal functor $\nec^{op}_{-}\rightarrow \mathbb{K}\Quiv_{\Ob(\mathcal{A})}$.

\begin{Not}
\begin{enumerate}[1.]
    \item Define a strong monoidal functor
    $$
    \overline{(-)}: \nec^{op}_{-}\to  \mathbb{K}\Quiv_{\Ob \mathcal{A}}: U\mapsto (s\mathcal{A})_{U}
    $$
    as follows. Since every active surjective necklace map is a unique wedge of surjective maps $e:\Delta^n \twoheadrightarrow \Delta^m$, it suffices to define:
    \[
    \overline{e}=\begin{cases}
    u: \underline{\mathbb{K}} \to \mathcal{A}_{0} = (s\mathcal{A})_1 & e: \Delta^1 \twoheadrightarrow \Delta^0\\
    \id_{\mathcal{A}_{n-1}} & e = \id_{\Delta^{n}}\\
    0 & \text{otherwise}
    \end{cases}
    \]
   
    \noindent Note that an active surjective map is spine collapsing (Definition \ref{Def:SpineCollapsing}) if and only if its image under $\overline{(-)}$ is nonzero.
    \item Let $f: T\rightarrow T'$ be a necklace map. For every $a,b\in \Ob(\mathcal{A})$, we define
    $$
    f^{*}: \bigoplus_{\substack{g: U\hookrightarrow T'\\ \text{injective}}}(s\mathcal{A})_{T'}(a,b)\rightarrow \bigoplus_{\substack{g: U\hookrightarrow T\\ \text{injective}}}(s\mathcal{A})_{T}(a,b): y\mapsto f^{*}(y)
    $$
    as follows. By Proposition \ref{Prop:Factorization}, we can factor $f\circ g$ uniquely as an active surjective necklace map $e$ followed by an injective necklace map $g'$. 
    $$
    \xymatrix{
    T\ar[r]^{f} & T'\\
    U\ar@{^{(}->}[u]^{g}\ar@{->>}[r]_{e} & U'\ar@{^{(}->}[u]_{g'}
    }
    $$
    Then set, for every injective necklace map $g: U\hookrightarrow T$,
    $$
    f^{*}(y)_{g} = \overline{e}(y_{g'}).
    $$
    Note that in particular, when $f$ is injective, we have $f^{*}(y)_{g} = y_{f\circ g}$.
\end{enumerate}
\end{Not}

\begin{Lem}\label{lemma:TANpreservedbystructuremaps}
Let $f: T\rightarrow T'$ be a necklace map and $a,b\in \Ob(\mathcal{A})$. If $y\in \bigoplus_{g: U\hookrightarrow T}(s\mathcal{A})_{U}(a,b)$ satisfies the (TAN) relations \eqref{equation:TAN}, then so does $f^{*}(y)$.
\end{Lem}
\begin{proof}
We show that the (TAN) relations hold for $f^{*}(y)$. Let $g: U\hookrightarrow T$ be an injective necklace map with $U = \Delta^{n_{1}}\vee\dots \vee\Delta^{n_{l}}$. As above, factor $f\circ g = g'\circ e$ where $e: U\twoheadrightarrow U'$ is active surjective and $g': U'\hookrightarrow T'$ is injective. We write $e = e_{1}\vee \dots\vee e_{l}$ with $e_{i}: \Delta^{n_{i}}\twoheadrightarrow \Delta^{m_{i}}$ and $n_{i} > 0$.

Let $1\leq k\leq l$. Given an inert necklace map $\nu: S\hookrightarrow \Delta^{n_{k}}$, factor $e_{k}\circ \nu = \nu'\circ e_{S}$ with $e_{S}$ active surjective and $\nu'$ injective. Then since $f\circ g\circ_{k} \nu = g'\circ e\circ_{k} \nu = (g'\circ_{k} \nu')\circ (e_{1}\vee \dots \vee e_{k-1}\vee e_{S}\vee e_{k+1}\vee \dots \vee e_{l})$, it follows from the uniqueness of the epi-mono factorization in $\nec$ that
$$
f^*(y)_{g\circ_k \nu}=(\overline{e_1}\otimes \cdots \otimes \overline{e_{k-1}}\otimes \overline{e_S} \otimes \overline{e_{k+1}}\otimes \cdots \otimes \overline{e_l})(y_{g'\circ_k \nu'})
$$
Moreover, it follows from \ref{eq:AooUnit} that
$$
m_{\ell(S)}\circ \overline{e_S} =
\begin{cases}
    m_{\ell(S)} & \text{if }e_k=\id_{\Delta^{n_{k}}}\\
    \id_{\mathcal{A}_{n_{k}-2}} & \text{if }(\nu = \nu_{1,n_{k}-1}, e_{k} = \sigma_0)\text{ or }(\nu = \nu_{n_{k}-1,1}, e_{k} = \sigma_{n_{k}-1})\\
  u & \text{if }\nu = \nu_{1,1}, e_{k} = \sigma_{0}\sigma_{1}: \Delta^2\twoheadrightarrow \Delta^0\\
    0 & \text{otherwise}
\end{cases}
$$
Note that in the second and third cases, $\nu' = \id_{\Delta^{n_{k}}}$.

On the other hand, for $0 < j < n_{k}$, factor $e_{k}\circ \delta_{j} = \delta'_{j}\circ e'_{k}$ with $e'_{k}$ active surjective and $\delta'_{j}$ injective. Then because $f\circ g\circ_{k} \delta_{j} = g'\circ e\circ_{k} \delta_{j} = (g'\circ_{k} \delta'_{j})\circ (e_{1}\vee \dots \vee e_{k-1}\vee e'_{k}\vee e_{k+1}\vee \dots \vee e_{l})$, it follows from the uniqueness of the epi-mono factorization in $\nec$ that
$$
f^{*}(y)_{g\circ_{k} \delta_{j}} = (\overline{e_{1}}\otimes \dots \otimes \overline{e_{k-1}}\otimes \overline{e'_{k}}\otimes \overline{e_{k+1}}\otimes \dots \otimes \overline{e_{l}})(y_{g'\circ_{k} \delta'_{j}})
$$
Moreover, it follows from the simplicial identities that
$$
\overline{e'_{k}} =
\begin{cases}
    \id_{\mathcal{A}_{n_{k}-2}} & \text{if }e_{k} = \id_{\Delta^{n_{k}}}\\
    \id_{\mathcal{A}_{n_{k}-2}} & \text{if }e_{k} = \sigma_{j-1}\text{ or }e_{k} = \sigma_{j}\\
    u & \text{if }e_{k} = \sigma_{0}\sigma_{1}: \Delta^{2}\twoheadrightarrow \Delta^{0}\\
    0 & \text{otherwise}
\end{cases}
$$
Note that in the second and third cases, $\delta'_{j} = \id_{\Delta^{n_{k}}}$.

We distinguish some cases:
\begin{itemize}
    \item Assume that $e_{k} = \id_{\Delta^{n_{k}}}$.\\
    Then with notations as above, we have $\nu' = \nu$ and $\delta'_{j} = \delta_{j}$ for any $\nu: S\hookrightarrow \Delta^{n_{k}}$ and $0 < j < n_{k}$. Recall that $\Bar{e}$ is non-zero if and only if $e$ is spine collapsing. Assuming $e$ is not spine collapsing, then both sides of the (TAN) relation simply vanish. If $e$ is spine collpasing, then $\dim(U)=\dim(U')$ and $\dim(U^{<k})=\dim(U'^{<k})$ for all $1\leq k \leq \ell(U)$. Consider equation \ref{equation:TANdecomp} for $y$ at $g'$ and $k$:
    $$
    \sum_{\substack{\nu: S\hookrightarrow \Delta^{n_{k}}\\ \text{inert}}}(-1)^{\epsilon(S)+\ell(S)\dim(U^{< k})} (\id^{\otimes k-1}\otimes m_{\ell(S)}\otimes \id^{\otimes l-k})(y_{g'\circ_{k} \nu}) = \sum_{j=1}^{n_{k}-1}(-1)^{j-1} y_{g'\circ_{k} \delta_{j}}
    $$
    By the previous calculations, applying $(\overline{e_{1}}\otimes \dots \otimes \id_{\mathcal{A}_{n_{k}-2}}\otimes \dots \otimes \overline{e_{l}})$ to this equation exactly recovers equation \eqref{equation:TANdecomp} for $f^{*}(y)$ at $g$ and $k$. 
    \item Assume that $e_{k} = \sigma_{0}$ or $e_{k} = \sigma_{n_{k}-1}$.\\
    In each case, by the previous calculations, both sides of equation \eqref{equation:TANdecomp} for $f^{*}(y)$ at $g$ and $k$ reduce to
    $$
    (-1)^s (\overline{e_{1}}\otimes \dots \otimes \id_{\mathcal{A}_{n_{k}-2}}\otimes \dots \otimes \overline{e_{l}})(y_{g'})
    $$
\noindent where $s=0$ for $e_k=\sigma_0$ and $s=n_k-2$ for $e_k=\sigma_{n_k-1}$
    \item Assume $e_{k} = \sigma_{i}$ for some $0 < i < n_{k}-1$.\\
    By the previous calculations, the left hand side of equation \eqref{equation:TANdecomp} at $g$ and $k$ is $0$, while the right hand side reduces to
    $$
    (-1)^{i-2} \left((\overline{e_{1}}\otimes \dots \otimes \id_{\mathcal{A}_{n_{k}-2}}\otimes \dots \otimes \overline{e_{l}})(y_{g'}) -(\overline{e_{1}}\otimes \dots \otimes \id_{\mathcal{A}_{n_{k}-2}}\otimes \dots \otimes \overline{e_{l}})(y_{g'}) \right)= 0
    $$

    \item Assume that $e_{k} = \sigma_{0}\sigma_{0}: \Delta^{2}\twoheadrightarrow \Delta^{0}$.\\
    By the previous calculations, both sides of equation \eqref{equation:TANdecomp} for $f^{*}(y)$ at $g$ and $k$ reduce to
    $$
     (\overline{e_{1}}\otimes \dots\otimes u\otimes \dots \otimes \overline{e_{l}})(y_{g'})
    $$

    \item In all other cases, both sides of equation \eqref{equation:TANdecomp} for $f^{*}(y)$ at $g$ and $k$ are $0$ by the previous calculations.
\end{itemize} 
\end{proof}

\begin{Prop}\label{proposition:Ainftynervestructuremaps}
    The assignments $T\mapsto N^{A_{\infty}}_{\mathbb{K}}(\mathcal{A})_{T}$ and $f\mapsto f^{*}$ define a functor $N^{A_{\infty}}_{\mathbb{K}}(\mathcal{A})_{\bullet}: \nec^{op}\rightarrow \mathbb{K}\Quiv_{\Ob(\mathcal{A})}$.
\end{Prop}
\begin{proof}
The functor is well-defined by Lemma \ref{lemma:TANpreservedbystructuremaps}. The fact that it preserves composition and identities immediately follows from the unique epi-mono factorization in $\nec$ and the functoriality of $\overline{(-)}$.
\end{proof}

Recall that by direct divisibility \ref{Lem:SpineCollandDim}, any injective necklace map $g: U\hookrightarrow T_{1}\vee T_{2}$ decomposes as $g = g_{1}\vee g_{2}$ for some unique $g_{i}: U_{i}\hookrightarrow T_{i}$. Consequently, we have a canonical quiver isomorphism
\begin{equation}\label{diagram:caniso}
    \bigoplus_{\substack{g_{2}: U\hookrightarrow T_{1}\\ \text{injective}}}(s\mathcal{A})_{U_{1}}\otimes_{\Ob(\mathcal{A})} \bigoplus_{\substack{g_{2}: U\hookrightarrow T_{2}\\ \text{injective}}}(s\mathcal{A})_{U_{2}}\xrightarrow{\simeq} \bigoplus_{\substack{g: U\hookrightarrow T_{1}\vee T_{2}\\ \text{injective}}}(s\mathcal{A})_{U}
\end{equation}
sending $x\otimes y\in \bigoplus_{U_{1}\hookrightarrow T_{1}}(s\mathcal{A})_{U_{1}}(a,c)\otimes \bigoplus_{g: U_{2}\hookrightarrow T_{2}}(s\mathcal{A})_{U_{2}}(c,b)$ to $(x_{g_{1}}\otimes y_{g_{2}})_{g}$.

\begin{Lem}\label{lemma:TANpreservesbycomp}
    Let $T_{1},T_{2}\in \nec$. The canonical quiver isomorphism \eqref{diagram:caniso} restricts to an isomorphism:
    $$
    N^{A_{\infty}}_{\mathbb{K}}(\mathcal{A})_{T_{1}}\otimes_{\Ob(\mathcal{A})} N^{A_{\infty}}_{\mathbb{K}}(\mathcal{A})_{T_{1}}\xrightarrow{\simeq} N^{A_{\infty}}_{\mathbb{K}}(\mathcal{A})_{T_{1}\vee T_{2}}
    $$
\end{Lem}
\begin{proof}
Let $a,b,c\in \Ob(\mathcal{A})$ and $x\in \bigoplus_{U_{1}\hookrightarrow T_{1}}(s\mathcal{A})_{U_{1}}(a,c)$, $y\in \bigoplus_{U_{2}\hookrightarrow T_{2}}(s\mathcal{A})_{U_{1}}(c,b)$. It will suffice to show that the collection $z = (x_{g_{1}}\otimes y_{g_{2}})_{g}$ satisfies the (TAN) relations \eqref{equation:TAN} if and only if both $x$ and $y$ satisfy them. 

 Given an injective necklace map $g: U\hookrightarrow T_{1}\vee T_{2}$, let $g_{i}: U_{i}\hookrightarrow T_{i}$ be unique such that $g = g_{1}\vee g_{2}$.
Take $1\leq k\leq \ell(U)$, and denote the $k$th bead of $U$ by $\Delta^{n_{k}}$. Then for any necklace map $h: V\rightarrow \Delta^{n_{k}}$, we have $g\circ_{k} h = (g_{1}\circ_{k} h)\vee g_{2}$ when $k\leq \ell(U_{1})$ and $g\circ_{k} h = g_{1}\vee (g_{2}\circ_{k} h)$ when $k > \ell(U_{1})$. It follows that the (TAN) relation for $z$ at $g$ is equivalent to the following $\ell(U)$ equations.

For all $1\leq k\leq \ell(U_{1})$, note that $\dim(U^{<k})=\dim(U_1^{<k})$:
$$
\left(\sum_{\substack{\nu: S\hookrightarrow \Delta^{n_{k}}\\ \text{inert}}}(-1)^{\epsilon(S)+\dim(U_1^{<k})\ell(S)}(\id^{\otimes k-1}\otimes m_{\ell(S)}\otimes \id^{\otimes \ell(U_{1})-k})(x_{g_{1}\circ_{k} \nu}) - \sum_{j=1}^{n_{k}-1}(-1)^{j-1} x_{g_{1}\circ_{k} \delta_{j}}\right)\otimes y_{g_{2}} = 0
$$
For all $ 1\leq k\leq \ell(U_2)$:
$$
x_{g_{1}}\otimes \left(\sum_{\substack{\nu: S\hookrightarrow \Delta^{n_{k}}\\ \text{inert}}}(-1)^{\epsilon(S)+\dim(U_2^{<k})\ell(S)}(\id^{\otimes k-1}\otimes m_{\ell(S)}\otimes \id^{\otimes \ell(U_2)-k})(y_{g_{2}\circ_{k} \nu}) - \sum_{j=1}^{n_{k}-1}(-1)^{j-1} y_{g_{2}\circ_{k} \delta_{j}}\right) = 0
$$
Since $\mathbb{K}$ is a field, it follows that $z$ satisfies the (TAN) relations if and only if ($x$ satisfies the (TAN) relations or $y = 0$) and ($x = 0$ and $y$ satisfies the (TAN) relations). Since $x = 0$ and $y = 0$ trivially satisfy the (TAN) relations, this concludes the proof.
\end{proof}
\begin{Rem}
The proof of the Lemma above is the only place in the article where we use that $\K$ is a field. We should thus remark that the construction of $N^{A_\infty}_{R}$ can be carried out for a general unital commutative ring $R$ provided we restrict ourselves to $A_\infty$-categories $\aaa$ such that $\aaa_n(x,y)$ are flat $R$-modules for any $x,y\in \Ob \aaa$ and $n\geq 0$.
\end{Rem}

\begin{Con}
\label{Con:NerveOnCats}
We construct a necklace category $N^{A_{\infty}}_{\mathbb{K}}(\mathcal{A})\in \mathbb{K}\Cat_{\nec}$ as follows. Its set of objects is given by $\Ob(\mathcal{A})$ and for each $a,b\in \Ob(\mathcal{A})$, its hom-object in $\Mod(\mathbb{K})^{\nec^{op}}$ is given by
$$
N^{A_{\infty}}_{\mathbb{K}}(\mathcal{A})_{\bullet}(a,b): \nec^{op}\rightarrow \Mod(\mathbb{K}): T\mapsto N^{A_{\infty}}_{\mathbb{K}}(\mathcal{A})_{T}(a,b), f\mapsto f^{*}
$$

Further, we define the composition maps of $N^{A_{\infty}}_{\mathbb{K}}(\mathcal{A})$ by the restricted isomorphisms of Lemma \ref{lemma:TANpreservesbycomp}. It is immediately clear that they are associative. They are also natural with respect to $T_{1}$ and $T_{2}$ by the uniqueness of the epi-mono factorization in $\nec$ and by the fact that the subcategories $\nec_{-}$ and $\nec_{+}$ are closed under $\vee$.
\end{Con}

\begin{Prop}
    The necklace category $N^{A_{\infty}}_{\mathbb{K}}(\mathcal{A})$ is a templicial vector space.
\end{Prop}
\begin{proof}
    Under the fully faithful embedding $(-)^{nec}: \ts\Mod(\mathbb{K})\hookrightarrow \mathbb{K}\Cat_{\nec}$, we may identify templicial vector spaces with necklace categories whose composition maps are isomorphisms (Proposition \ref{Prop:tempnec}). Hence, this is immediate from the construction.
\end{proof}

\begin{Def}
    We refer to the templicial vector space $N^{A_{\infty}}_{\mathbb{K}}(\mathcal{A})$ as the \emph{templicial $A_{\infty}$-nerve} of $\mathcal{A}$.
\end{Def}

\begin{Exs}
Let us make the face, degeneracy and comultiplication maps in the templicial vector space $N^{A_\infty}_{\mathbb{K}}(\aaa)$ more explicit.

Since the face maps $d_{j}$ and comultiplication maps $\mu_{p,q}$ are parametrized by injective necklace maps ($\delta_{j}$ and $\nu_{p,q}$ respectively), they are simply described by
\begin{itemize}
    \item for all $0 < j < n$:
    $$
    d_{j}: N^{A_\infty}_\mathbb{K}(\aaa)_{n}\rightarrow N^{A_\infty}_\mathbb{K}(\aaa)_{n-1}: y\mapsto \left(y_{\delta_{j}\circ g}\right)_{g: U\hookrightarrow \Delta^{n-1}}
    $$
    \item for all $p,q > 0$:
    $$
    \mu_{p,q}: N^{A_\infty}_\mathbb{K}(\aaa)_{p+q}\rightarrow N^{A_\infty}_\mathbb{K}(\aaa)_{p}\otimes_{\Ob \aaa}N^{A_\infty}_\mathbb{K}(\aaa)_{q}: y\mapsto (y_{\nu_{p,q}\circ g})_{g: U\hookrightarrow \Delta^{p}\vee \Delta^{q}}
    $$
    \noindent where we've identified $N^{A_\infty}_\mathbb{K}(\aaa)_{p}\otimes_{\Ob \aaa}N^{A_\infty}_\mathbb{K}(\aaa)_{q}$ with $N^{A_{\infty}}_\mathbb{K}(\aaa)_{\Delta^{p}\vee \Delta^{q}}$ under the isomorphism of Lemma \ref{lemma:TANpreservesbycomp}.
\end{itemize}
The degeneracy maps $s_{i}$ on the other hand are parametrized by the surjective maps $\sigma_{i}$. Given an injective necklace map $g: (U,p)\hookrightarrow \Delta^{n+1}$, factor $\sigma_{i}g = g'e$ with $g'$ injective and $e$ active surjective. Then $e = \id_{U}$ if and only if $g: [p]\hookrightarrow [n+1]$ factors through $\delta_{i}$ or $\delta_{i+1}$, which is equivalent to $\{i,i+1\}\not\subseteq g([p])$. Further, if $e\neq \id_{U}$, then it is spine collapsing if and only if $U = U_{1}\vee \Delta^{1}\vee U_{2}$ and the joints of $\Delta^{1}$ are sent to $i$ and $i+1$ by $g$. This is equivalent to $\{i,i+1\}\subseteq g(U)$. In this case, $e = \id\vee \sigma_{0}\vee \id$ and hence we find
\begin{itemize}
    \item for all $0\leq i\leq n$:
    $$
    s_{i}: N^{A_\infty}_\mathbb{K}(\aaa)_n\rightarrow N^{A_\infty}_\mathbb{K}(\aaa)_{n+1}: y\mapsto s_{i}(y)
    $$
    where for all injective $g: U\hookrightarrow \Delta^{n+1}$:
    $$
    s_{i}(y)_{g} =
    \begin{cases}
    y_{\sigma_i\circ g} & \{i,i+1\}\not\subset g([p])\\
    (\id \otimes_{\Ob \aaa} u \otimes_{\Ob \aaa} \id)(y_{g'}) & \{i,i+1\}\subseteq g(U)\\
    0 & \text{otherwise}\\
    \end{cases}
    $$
\end{itemize}
\end{Exs}

\subsection{Templicial maps associated to $A_{\infty}$-functors}

Given an $A_{\infty}$-functor, we construct a necklicial functor between the associated templicial $A_{\infty}$-nerves. Since $(-)^{nec}$ is fully faithful, this is in fact a templicial map. In this subsection, we show that this construction defines a functor
$$
N^{A_{\infty}}_{\mathbb{K}}: \suAinftyCat \rightarrow \ts\Mod(\mathbb{K})
$$
For now, let us fix an $A_{\infty}$-functor $f = (f_{0},f_{1},f_{2},\dots): \mathcal{A}\rightarrow \mathcal{B}$ between (small) $\mathcal{A}_{\infty}$-categories, where $f_{0}: \Ob(\mathcal{A})\rightarrow \Ob(\mathcal{B})$ denotes the object map.

\begin{Not}\label{notation:tempmapoffunctor}
    Given $T\in \nec$ and $a,b\in \Ob(\mathcal{A})$, we define a $\mathbb{K}$-linear map
    $$
    f_{T}: \bigoplus_{\substack{g: U\hookrightarrow T\\ \text{injective}}}(s\mathcal{A})_{U}(a,b)\rightarrow \bigoplus_{\substack{g: U\hookrightarrow T\\ \text{injective}}}(s\mathcal{B})_{U}(f_{0}(a),f_{0}(b)): y\mapsto f_{T}(y)
    $$
    as follows. Let $y\in \bigoplus_{U\hookrightarrow T}(s\mathcal{A})_{U}(a,b)$, and let $g: U\hookrightarrow T$ be an injective necklace map. Then we set
    $$
    f_{T}(y)_g = \sum_{\substack{\nu: S \hookrightarrow U\\ \text{inert}}}(-1)^{\epsilon(S_1)+\cdots + \epsilon(S_l)+\varphi(\nu)} (f_{\ell(S_1)}\otimes \cdots \otimes f_{\ell(S_l)})(y_{g\circ \nu})\in (s\mathcal{B})_{U}(f_{0}(a),f_{0}(b))
    $$
    where for $U = \Delta^{n_{1}}\vee \dots\vee \Delta^{n_{l}}$, we've written $S_{1},\dots,S_{l}$ for the unique necklaces such that $\nu$ is the wedge product of inert maps $S_{i}\hookrightarrow \Delta^{n_{i}}$. Note that then $f_{\ell(S_{i})}: (s\mathcal{A})_{S_{i}}\rightarrow f^{*}_{0}(s\mathcal{B})_{n_{i}}$ so that $f_{T}$ is well-defined.
    
    In particular, if $T = \Delta^{0}$, we have that $f_{\Delta^{0}}$ is given by the identity on $\mathbb{K} = (s\mathcal{A})_{\Delta^{0}}(a,a) = (s\mathcal{B})_{\Delta^{0}}(f_{0}(a),f_{0}(a))$.
\end{Not}

\begin{Lem}\label{lemma:inertmapscombinatorics}
Let $n,r > 0$ be integers. Consider the following sets
\begin{align*}
    \mathcal{S}_{1} &= \left\lbrace (\nu,\mu_{1},\dots,\mu_{r})\,\middle\vert\,\nu: \Delta^{m_{1}}\vee \dots \vee \Delta^{m_{r}}\hookrightarrow \Delta^{n}, \mu_{j}: T_{j}\hookrightarrow \Delta^{m_{j}}\text{ inert}\right\rbrace\\
    \mathcal{S}_{2} &= \left\lbrace (\eta, i_{1},\dots, i_{r})\,\middle\vert\, \eta: T\hookrightarrow \Delta^{n}\text{ inert}, i_{j} > 0, i_{1} + \dots + i_{r} = \ell(T)\right\rbrace\\
    \mathcal{T}_{1} &= \left\lbrace (\mu,k,\nu)\,\middle\vert\, \mu: V\hookrightarrow \Delta^{n}\text{ inert}, 1\leq k\leq \ell(V), \nu: W\hookrightarrow \Delta^{n_{k}}\text{ inert}\right\rbrace\\
    \mathcal{T}_{2} &= \left\lbrace (\eta,r,s,t)\,\middle\vert\, \eta: T\hookrightarrow \Delta^{n}\text{ inert}, r,t\geq 0, s > 0, r + s + t = \ell(T)\right\rbrace\\
    \mathcal{U}_{1} &= \left\lbrace (\mu,j)\,\middle\vert\, \mu: V\hookrightarrow \Delta^{n}\text{ inert}, 1\leq j\leq \dim(V)\right\rbrace\\
    \mathcal{U}_{2} &= \left\lbrace (\mu',i)\,\middle\vert\, \mu': V'\hookrightarrow \Delta^{n-1}\text{ inert}, 0 < i < n\right\rbrace.
\end{align*}
Then the following maps are bijections
\begin{enumerate}[(a)]
    \item $\mathcal{S}_{1}\rightarrow \mathcal{S}_{2}: (\nu,\mu_{1},\dots,\mu_{r})\mapsto (\nu\circ (\mu_{1}\vee \dots \vee \mu_{r}), \ell(T_{1}),\dots, \ell(T_{r}))$
    \item $\mathcal{T}_{1}\rightarrow \mathcal{T}_{2}: (\mu,k,\nu)\mapsto (\mu\circ_{k} \nu,k-1,\ell(W),\ell(V)-k)$
    \item $\mathcal{U}_{1}\rightarrow \mathcal{U}_{2}: (\mu,j)\mapsto (\mu',i_{j})$
\end{enumerate}
where in $(c)$, $\mu': V\hookrightarrow \Delta^{n-1}$ is the unique inert map such that $\mu\circ \delta_{j} = \delta_{i_{j}}\circ \mu'$, and we've written $V^{c} = \{i_{1} < \dots < i_{d}\}$ with $d = \dim(V)$.
\end{Lem}
\begin{proof}
    This is a quick
    verification.
\end{proof}

\begin{Prop}\label{proposition:TANpreservedbyfunctor}
Let $T$ be a necklace and $a,b\in \Ob(\mathcal{A})$. If $y\in \bigoplus_{U\hookrightarrow T}(s\mathcal{A})_{U}(a,b)$ satisfies the (TAN) relations, then so does $f_{T}(y)$.
\end{Prop}
\begin{proof}
Let $g: U\hookrightarrow T$ be an injective necklace map. Since for any inert necklace map $\nu: S\hookrightarrow U$, we have $g^{*}(y)_{\nu} = y_{g\circ\nu}$, it follows that $f_{T}(y)_{g} = f_{U}(g^{*}(y))_{\id_{U}}$. Hence by Lemma \ref{lemma:TANpreservedbystructuremaps}, it suffices to check that $f_{T}(y)$ satisfies the (TAN) relation at $\id_{T}$. Suppose $T = T_{1}\vee T_{2}$, then we may assume that $y = y^{(1)}\otimes y^{(2)}$ for some $y^{(i)}\in N^{A_{\infty}}_{k}(\mathcal{A})_{T_{i}}$ by Lemma \ref{lemma:TANpreservesbycomp}. Since for any inert maps $\nu_{i}: S_{i}\hookrightarrow T_{i}$, we have $y_{\nu_{1}\vee \nu_{2}} = y^{(1)}_{\nu_{1}}\otimes y^{(2)}_{\nu_{2}}$, it follows that $f_{T}(y)_{\id_{T}} = f_{T_{1}}(y^{(1)})_{\id_{T_{1}}}\otimes f_{T_{2}}(y^{(2)})_{\id_{T_{2}}}$. Hence, it suffices to check that $f_{\Delta^{n}}(y)$ satisfies (TAN) relation \eqref{equation:TAN} at $\id_{\Delta^{n}}$ for all $n > 0$.
\par The left hand side of the (TAN) relation of $f_{\id_{\Delta^{n}}}(y)$ at $\id_{\Delta^{n}}$ is
\begin{align*}
    \sum_{\substack{\nu: S\hookrightarrow \Delta^{n}\\ \text{inert}}}&(-1)^{\epsilon(S)} m_{\ell(S)}(f_{\id_{\Delta^{n}}}(y)_{\nu}) = \\&=\sum_{\substack{\nu: S\hookrightarrow \Delta^{n}\\ \mu: T\hookrightarrow S\\ \text{inert}}}(-1)^{\epsilon(S)+\epsilon(T_1)+\cdots +\epsilon(T_l)+\varphi(\mu)} m_{r}(f_{\ell(T_{1})}\otimes \dots \otimes f_{\ell(T_{r})})(y_{\nu\circ \mu})
\end{align*}

where we've written $\mu = \mu_{1}\vee \dots \vee \mu_{r}$ for unique inert maps $\mu_{i}: T_{i}\hookrightarrow \Delta^{n_{i}}$. By Lemma \ref{lemma:inertmapscombinatorics}(a), Lemma \ref{Lem:MagicFormula} and since $f$ is an $A_{\infty}$-functor, this equals
\begin{align*}
 \sum_{\substack{\eta: T\hookrightarrow \Delta^{n}\\ \text{inert}\\ r > 0, i_{1},\dots ,i_{r} > 0\\ i_{1} + \dots + i_{r} = \ell(T)}} &(-1)^{\epsilon(T)+\epsilon(i_1,...,i_r)}  m_{r}(f_{i_{1}}\otimes \dots \otimes f_{i_{r}})(y_{\eta}) =\\&= \sum_{\substack{\eta: T\hookrightarrow \Delta^{n}\\ \text{inert}\\ r,t\geq 0, s > 0\\ r + s + t = \ell(T)}}(-1)^{\epsilon(T)+r+st+1} f_{r + 1 + t}(\id^{\otimes r}\otimes m_{s}\otimes \id^{\otimes t})(y_{\eta})   
\end{align*}

Then by Lemma \ref{lemma:inertmapscombinatorics}(b) and the (TAN) relation of $y$ at $\mu$, this is equal to
\begin{align*}
    \sum_{\substack{\mu: V\hookrightarrow \Delta^{n}\\ \text{inert}}}\sum_{\substack{1\leq k\leq l\\ \nu: W\hookrightarrow \Delta^{n_{k}}\\ \text{inert}}}&(-1)^{\epsilon(V)-\ell(V)+\epsilon(W)+\varphi_k(\nu)}f_{l}(\id^{\otimes k-1}\otimes m_{\ell(W)}\otimes \id^{\otimes l-k})(y_{\mu\circ_{k} \nu}) = \\=&\sum_{\substack{\mu: V\hookrightarrow \Delta^{n}\\ \text{inert}}}\sum_{j = 1}^{d} (-1)^{\epsilon(V)-\ell(V)+j}f_{\ell(V)}(y_{\mu\circ \delta_{i_{j}}})
\end{align*}
where we've written $V = \Delta^{n_{1}}\vee \dots \vee \Delta^{n_{l}}$ and $V^{c} = \{i_{1} < \dots < i_{d}\}$ so that $l = \ell(V)$ and $d = \dim(V)$ for an inert $\mu: V\hookrightarrow \Delta^{n}$. Finally, by Lemma \ref{lemma:inertmapscombinatorics}(c), we find that this is equal to the right hand side of the (TAN) relation of $f_{\id_{\Delta^{n}}}(y)$ at $\id_{\Delta^{n}}$:
$$
\sum_{i=1}^{n-1}\sum_{\mu': V'\hookrightarrow \Delta^{n-1}} (-1)^{\epsilon(V)+i-1}f_{\ell(V')}(y_{\delta_{i}\circ \mu'}) = \sum_{i=1}^{n}(-1)^{i-1}f_{\id_{\Delta^{n}}}(y)_{\delta_{i}}
$$
\end{proof}

\begin{Lem}\label{lemma:functorisnatural}
The morphisms $(f_{T})_{T\in \nec}$ define a natural transformation $N^{A_{\infty}}_{\mathbb{K}}(\mathcal{A})_{\bullet}\rightarrow f^{*}_{0}N^{A_{\infty}}_{\mathbb{K}}(\mathcal{B})_{\bullet}$ between functors $\nec^{op}\rightarrow \mathbb{K}\Quiv_{\Ob(\mathcal{A})}$.
\end{Lem}
\begin{proof}
Let $h: T\rightarrow T'$ and $g: U\hookrightarrow T$ be necklace maps with $g$ injective. We wish to show that for all $a,b\in \Ob(\mathcal{A})$ and $y\in N^{A_{\infty}}_{\mathbb{K}}(\mathcal{A})_{T'}(a,b)$, we have
\begin{equation}\label{equation:functorisnatural}
    h^{*}(f_{T'}(y))_{g} = f_{T}(h^{*}(y))_{g}  
\end{equation}
As before, for any $z\in N^{A_{\infty}}_{\mathbb{K}}(\mathcal{A})_{T}(a,b)$ we have $f_{T}(z)_{g} = f_{U}(g^{*}(z))_{\id_{U}}$. It follows that we may assume $h = e$ is an epimorphism and $g = \id_{T}$. If $T = T_{1}\vee T_{2}$, then we may assume by Lemma \ref{lemma:TANpreservesbycomp} that $y = y^{(1)}\otimes y^{(2)}$ for some $y^{(i)}\in N^{A_{\infty}}_{k}(\mathcal{A})_{T_{i}}$. We can write $e = e^{(1)}\vee e^{(2)}$ for some epimorphisms $e^{(i)}: T_{i}\twoheadrightarrow T'_{i}$. Then as before, we have $f_{T'}(y)_{\id_{T}} = f_{T'_{1}}(y^{(1)})_{\id_{T'_{1}}}\otimes f_{T'_{2}}(y^{(2)})_{\id_{T_{2}}}$ and it follows that we may assume that $e: \Delta^{n}\twoheadrightarrow \Delta^{m}$ with $n > 0$, $m\geq 0$ and that $g = \id_{\Delta^{n}}$.

The rest of the proof follows in a similar way to the proof of Lemma \ref{lemma:TANpreservedbystructuremaps}. We distinguish some cases:
\begin{enumerate}[1.]
    \item $e = \id_{\Delta^{n}}$: Then equation \eqref{equation:functorisnatural} holds trivially.
    \item $e = \sigma_{0}: \Delta^{1}\twoheadrightarrow \Delta^{0}$: Then equation \eqref{equation:functorisnatural} is equivalent to $1_{f_{0}(a)} = f_{1}(1_{a})$ if $a = b$, and trivially true otherwise. Note that this holds since $f$ is a functor of strictly unital $A_{\infty}$-categories.
    \item Otherwise, let $\nu: S\hookrightarrow \Delta^{n}$ be an inert necklace map. We can factor $e\circ \nu = \nu'\circ e_{S}$ for some unique epimorphism $e_{S}: S\twoheadrightarrow S'$ and injective necklace map $\nu': S'\hookrightarrow \Delta^{m}$. Then $e_{S} = e_{1}\vee \dots \vee e_{\ell(S)}$ for some epimorphisms $e_{i}: \Delta^{n_{i}}\twoheadrightarrow \Delta^{m_{i}}$ with $n_{i} > 0$ and $m_{i}\geq 0$. If $\ell(S)=1$, then $\nu=\id_{\Delta^n}$, $e_S=e$, $\nu'=\id_{\Delta^m}$ and this term vanishes as $e\neq \id, \sigma_0$. Assume then that $\ell(S)>1$.
    Also as $e\neq \id_{\Delta^{n}}$, we have $e_{j}\neq \id_{\Delta^{n_{j}}}$ for at least one $j\in \{1,\dots ,\ell(S)\}$. Note that
    \begin{itemize}
        \item if $e_{j} = \sigma_{0}: \Delta^{1}\twoheadrightarrow \Delta^{0}$, then $\overline{e_{i}} = u$. Since $e\neq \sigma_{0}$ and thus $f_{\ell(S)}(\overline{e}_{1}\otimes \dots \otimes \overline{e}_{\ell(S)}) = 0$.
        \item if $e_{j}\neq \sigma_{0}$, then $\overline{e}_{j} = 0$ and thus $f_{\ell(S)}(\overline{e}_{1}\otimes \dots \otimes \overline{e}_{\ell(S)}) = 0$ as well.
    \end{itemize}
    If follows that both sides of equation \eqref{equation:functorisnatural} reduce to $0$.
\end{enumerate}
\end{proof}

\begin{Con}\label{Con:NerveOnMaps}
    Let $f: \mathcal{A}\rightarrow \mathcal{B}$ be an $A_{\infty}$-functor. We construct a necklicial functor $N^{A_{\infty}}_{\mathbb{K}}(f): N^{A_{\infty}}_{\mathbb{K}}(\mathcal{A})^{nec}\rightarrow N^{A_{\infty}}_{\mathbb{K}}(\mathcal{B})^{nec}$ between necklicial categories as follows.

    The object map of $N^{A_{\infty}}_{\mathbb{K}}(f)$ is simply given by $f_{0}: \Ob(\mathcal{A})\rightarrow \Ob(\mathcal{B})$. Then for any $a,b\in \Ob(\mathcal{A})$, we define $N^{A_{\infty}}_{\mathbb{K}}(f)$ on hom-objects by the natural transformation $N^{A_{\infty}}_{\mathbb{K}}(f)_{a,b}: N^{A_{\infty}}_{\mathbb{K}}(\mathcal{A})_{\bullet}(a,b)\rightarrow N^{A_{\infty}}_{\mathbb{K}}(\mathcal{B})_{\bullet}(f_{0}(a),f_{0}(b))$ of \ref{lemma:functorisnatural}.

    Finally, $N^{A_{\infty}}_{\mathbb{K}}(f)$ preserves the identities by definition. Let $g: U\hookrightarrow T_{1}\vee T_{2}$ be an injective necklace map with $g_{i}: U_{i}\hookrightarrow T_{i}$ unique such that $g = g_{1}\vee g_{2}$, and $l_{i} = \ell(U_{i})$. Note that for all inert maps $\nu_{k}: S_{k}\hookrightarrow \Delta^{n_{k}}$ into the beads of $U$, we have
    $$
    g\circ (\nu_{1}\vee \dots \vee \nu_{l_{1} + l_{2}}) = (g_{1}\circ (\nu_{1}\vee \dots \vee \nu_{l_{1}}))\vee (g_{2}\circ (\nu_{l_{1}+1}\vee \dots \vee \nu_{l_{1}+l_{2}}))
    $$
    By the definition, this shows $N^{A_{\infty}}_{\mathbb{K}}(f)$ also preserves the composition.
\end{Con}

\begin{Prop}
\label{Prop:NerveWellDefined}
The assignments $\mathcal{A}\mapsto N^{A_{\infty}}_{\mathbb{K}}(\mathcal{A})$ and $f\mapsto N^{A_{\infty}}_{\mathbb{K}}(f)$ of Constructions \ref{Con:NerveOnCats} and \ref{Con:NerveOnMaps} define a functor $N^{A_{\infty}}_{\mathbb{K}}: \suAinftyCat\rightarrow \ts\Mod(\mathbb{K})$.
\end{Prop}
\begin{proof}
Let $\mathcal{A}$ be a strictly unital $A_{\infty}$-category and consider the identity $\id_{\mathcal{A}}$ on $\mathcal{A}$. Then $(\id_{\mathcal{A}})_{n}: \mathcal{A}^{\otimes n}_{\bullet}\rightarrow \mathcal{A}_{\bullet}$ is zero for all $n > 1$ and the identity on the complex $\mathcal{A}_{\bullet}$ otherwise. It follows from Notation \ref{notation:tempmapoffunctor} that $(\id_{\mathcal{A}})_{T}: N^{A_{\infty}}_{\mathbb{K}}(\mathcal{A})_{T}\rightarrow N^{A_{\infty}}_{\mathbb{K}}(\mathcal{A})_{T}$ is the identity as well for each $T\in \nec$.

Let $f: \mathcal{A}\rightarrow \mathcal{B}$ and $g: \mathcal{B}\rightarrow \mathcal{C}$ be $A_{\infty}$-functors. We wish to show that for all injective necklace maps $h: U\hookrightarrow T$ and $a,b\in \Ob(\mathcal{A})$, $y\in N^{A_{\infty}}_{\mathbb{K}}(\mathcal{A})_{T}(a,b)$:
$$
g_{T}(f_{T}(y))_{h} = (g\circ f)_{T}(y)_{h}
$$
As before, it suffices to check the case where $T = \Delta^{n}$ and $h = \id_{\Delta^{n}}$. Then:
\begin{align*}
  g_{\Delta^{n}}(f_{\Delta^{n}}(y))_{\id_{\Delta^{n}}} &= \sum_{\substack{\nu: S\hookrightarrow \Delta^{n}\\ \text{inert}}}(-1)^{\epsilon(S)} g_{\ell(S)}(f_{T}(y)_{\nu}) \\&= \sum_{\substack{\nu: S\hookrightarrow \Delta^{n}, \ell(S)=r\\ \mu: T\hookrightarrow S\\ \text{inert}}}(-1)^{\epsilon(S)+\epsilon(T_1)+\cdots + \epsilon(T_r)+\varphi(\mu)} g_{r}(f_{\ell(T_{1})}\otimes \dots \otimes f_{\ell(T_{r})})(y_{\nu\circ \mu})  
\end{align*}

\noindent where we’ve written $\mu = \mu_1\vee \dots \vee\mu_r$ for unique inert maps $\mu_{i}: T_{i}\hookrightarrow \Delta^{n_{i}}$. By Lemma \ref{lemma:inertmapscombinatorics}(a), Lemma \ref{Lem:MagicFormula} and the definition of the composition of $A_{\infty}$-functors, this equals
$$
\sum_{\substack{\eta: T\hookrightarrow \Delta^{n}\\ \text{inert}\\ r > 0, i_{1},\dots,i_{r} > 0\\ i_{1} + \dots + i_{r} = \ell(T)}}\!\!\!\!\!(-1)^{\epsilon(T)+\epsilon_g(i_1,...,i_r)} g_{r}(f_{i_{1}}\otimes \dots \otimes f_{i_{r}})(y_{\eta}) = \sum_{\substack{\eta: T\hookrightarrow \Delta^{n}\\ \text{inert}}}(-1)^{\epsilon(T)} (g\circ f)_{T}(y_{\eta}) = (g\circ f)_{\Delta^{n}}(y)_{\id_{\Delta^{n}}}
$$
\end{proof}

\begin{Def}
We call the functor $N^{A_{\infty}}_{\mathbb{K}}: \suAinftyCat\rightarrow \ts\Mod(\mathbb{K})$ the \emph{templicial $A_{\infty}$-nerve functor}.
\end{Def}

\section{Main Results}
\label{Sec:MainResults}

\par In this section we prove some important properties of the construction $N^{A_\infty}_\K$.
\begin{Thm}
\label{Thm:MapIntoN}
Let $(X,O)$ be a templicial vector space and $\mathcal{A}$ an $A_\infty$-category. Then a map $(\alpha,f): X\to N^{A_\infty}_\mathbb{K}(A)$ is equivalent to the following data
\begin{itemize}
    \item a map of sets $f:O\to \Ob(\mathcal{A})$
    \item for all $n>0$, a map of $\mathbb{K}$-quivers $\beta_n: X_n \to f^* \mathcal{A}_{n-1}$
\end{itemize}
\noindent subject to the following conditions
\begin{itemize}
    \item for all $n\geq 0$ and $0\leq j \leq n$:
    \begin{equation}\label{equation:MapIntoNdeg}
        \beta_{n+1}\circ s_j^X=\begin{cases}
            u_\aaa & n=0\\
            0 & n>0\\
        \end{cases}
    \end{equation}
    \item for all $n>0$:
    \begin{equation}
    \label{eq:UP}
        \partial\beta_n = \sum_{j=1}^{n-1}(-1)^{j-1} \beta_{n-1} d_j^X - \sum_{\substack{i_1+\cdots+i_l=n\\i_{j} > 0, l\geq 2}} (-1)^{1+\epsilon_g(i_1,...,i_l)} m_l \left(\beta_{i_1}\otimes_T\cdots
    \otimes_T \beta_{i_l}\right)\mu^X_{i_1,...,i_l}
    \end{equation}
\end{itemize}
\noindent where $m_{n}$ denotes the multiplications of $f^{*}\mathcal{A}$, see Example \ref{Ex:AooCatPullback}.
\end{Thm}
\begin{proof}
Without loss of generality, we may replace $\mathcal{A}$ by $f^{*}\mathcal{A}$ and assume that $f$ is the identity on $T$. For each injective necklace map $g: U\hookrightarrow \Delta^{n}$ with $n > 0$, let us denote
$$
\alpha_n^{g}: X_n \xrightarrow{\alpha_n} N^{A_\infty}_k(\aaa)_{n}\xrightarrow{\pi_g} (s\aaa)_U
$$
Consider the fully faithful embedding $(-)^{nec}: \ts\Vect(\mathbb{K})\hookrightarrow \mathbb{K}\Cat_{\nec}$ (\ref{eq:nec}). Then the compatibility of $\alpha$ with the inner face maps and the comultiplications is equivalent to having, for all injective necklace maps $h: T\hookrightarrow \Delta^{n}$:
\begin{equation}\label{equation:MapIntoN1}
   h^{*}\alpha_{n} = \alpha^{nec}_{T}X^{nec}(h)\quad \Longleftrightarrow\quad \forall g: U\hookrightarrow T: \alpha^{hg}_{n} = (\alpha^{g_{1}}_{n_{1}}\otimes \dots \otimes \alpha^{g_{k}}_{n_{k}})X^{nec}(h) 
\end{equation}
where $T = \Delta^{n_{1}}\vee \dots \vee \Delta^{n_{k}}$ and the injective maps $g_{i}: U_{i}\hookrightarrow \Delta^{n_{i}}$ are unique such that $g = g_{1}\vee \dots \vee g_{n}$. In particular for $g = \id_{T}$, we find that
\begin{equation}\label{equation:MapIntoN2}
    \alpha^{h}_{n} = (\alpha^{\id_{\Delta^{n_{1}}}}_{n_{1}}\otimes \dots \otimes \alpha^{\id_{\Delta^{n_{k}}}}_{n_{k}})X^{nec}(h)
\end{equation}
Note that since $X^{nec}(hg) = (X^{nec}(g_{1})\otimes \dots \otimes X^{nec}(g_{k}))X^{nec}(h)$, it follows that the assignment $(\alpha^{h}_{n})_{n,h}\mapsto \alpha^{\id_{\Delta^{n}}}_{n}$ defines a bijection between the set of collections of quiver maps $(\alpha^{g}_{n}: X_{n}\rightarrow (s\mathcal{A})_{U})_{\substack{n > 0\\ g: U\hookrightarrow \Delta^{n}}}$ satisfying \eqref{equation:MapIntoN1}, and the set of collections of quiver maps
$$
(\beta_{n}: X_{n}\rightarrow \mathcal{A}_{n-1})_{n > 0}
$$
We will finish the proof by showing that
\begin{enumerate}[(i)]
    \item the compatibility of $\alpha$ with the degeneracy maps is equivalent to equation \eqref{equation:MapIntoNdeg} for $\beta_{n} = \alpha^{id_{\Delta^{n}}}_{n}$.
    \item the fact that $\alpha_{n}(x)$ satisfies the the (TAN) relations for all $n > 0$ and $x\in X_{n}$ is equivalent to equation \eqref{eq:UP} for $\beta_{n} = \alpha^{\id_{\Delta^{n}}}_{n}$.
\end{enumerate}

To prove $(i)$, note that $\sigma^{*}_{i}\alpha_{n} = \alpha_{n+1}s^{X}_{i}$ for all $1\leq i\leq n$ if and only if for all injective $g: U\hookrightarrow \Delta^{n+1}$ we have $\overline{e}\alpha^{g'}_{n} = \alpha^{g}_{n+1}s^{X}_{i}$, where we've factored $\sigma_{i} g = g'e$ with $g'$ injective and $e$ an epimorphism. Using \eqref{equation:MapIntoN2} and the fact that $X(e)X(g') = X(g)s^{X}_{i}$, we see that this is further equivalent to having $\overline{\sigma_{i}}\alpha^{\id_{\Delta^{n+1}}}_{n+1} = \alpha^{\id_{\Delta^{n}}}_{n}s^{X}_{i}$. Finally, $\overline{\sigma_{i}} = 0$ unless $n = i = 0$ and thus this recovers precisely equation \eqref{equation:MapIntoNdeg}.

To prove $(ii)$, note that \eqref{eq:UP} expresses precisely the (TAN) relation for $\alpha_{n}$ at $\id_{\Delta^{n}}$ for all $n > 0$. All the other (TAN) relations follow from these ones, so they are superfluous. Let $g: U= \Delta^{n_1}\vee\cdots\vee \Delta^{n_l}\hookrightarrow \Delta^n$ be injective. Then for all $1\leq i\leq n$:
\begin{align*}
    &\sum_{\substack{\nu: S\hookrightarrow \Delta^{n_{i}}\\ \text{inert}}}(-1)^{\epsilon(S)+\dim(U^{<i})\ell(S)} (\id\otimes m_{\ell(S)}\otimes \id)(\alpha_n^{g\circ_i \nu })\\
    =& \sum_{\substack{\nu: S\hookrightarrow \Delta^{n_{i}}\\ \text{inert}}}(-1)^{\epsilon(S)+\dim(U^{<i})\ell(S)} (\id \otimes m_{\ell(S)}\otimes \id)(\beta_{n_1}\otimes \cdots \otimes \beta_S \otimes \cdots \otimes \beta_{n_l})X^{nec}(g\circ_i \nu)\\
    =& \sum_{\substack{\nu: S\hookrightarrow \Delta^{n_{i}}\\ \text{inert}}}(-1)^{\epsilon(S)} (\beta_{{n_1}}\otimes \cdots \otimes m_{\ell(S)} \beta_S X^{nec}(\nu) \otimes \cdots \otimes \beta_{{n_l}})X^{nec}(g)\\
\end{align*}
\begin{align*}
    =& \sum_{k=1}^{n_i-1}(-1)^{k-1} 
    (\beta_{n_1}\otimes \cdots \otimes \beta_{n-1}X^{nec}(\delta_k) \otimes \cdots \otimes \beta_{n_l})X^{nec}(g)\\
    =& \sum_{k=1}^{n_i-1}(-1)^{k-1} 
    (\beta_{n_1}\otimes \cdots \otimes \beta_{n-1} \otimes \cdots \otimes \beta_{n_l})X^{nec}(g \circ_i \delta_k)\\
    =& \sum_{k=1}^{n_i-1}(-1)^{k-1} 
    \alpha_{n-1}^{g\circ \delta_k}
\end{align*}
where $\beta_{p} = \alpha^{\id_{\Delta^{p}}}_{p}$ and $\beta_{S} = \beta_{m_{1}}\otimes \dots \otimes\beta_{m_{k}}$ for $S = \Delta^{m_{1}}\vee \dots \vee \Delta^{m_{k}}$.
\end{proof}

\begin{Cor}
\label{Cor:UnderlyingSSet}
    For any $A_\infty$-category $\mathcal{A}$,
    \[
    \tilde{U}(N^{A_\infty}_\mathbb{K}(\mathcal{A}))\cong N^{A_\infty}(\mathcal{A})
    \]
\end{Cor}
\begin{proof}
Let $n\geq 0$ be an integer. By the adjunction $\tilde{F}\dashv \tilde{U}$, an $n$-simplex of $\tilde{U}(N^{A_{\infty}}_{\mathbb{K}}(\mathcal{A}))$ is equivalent to a templicial map $\tilde{F}(\Delta^{n})\rightarrow N^{A_{\infty}}_{\mathbb{K}}(\mathcal{A})$. By the previous proposition, this consists of
\begin{itemize}
    \item a choice of objects $A_{0},\dots, A_{n}\in \Ob(\mathcal{A})$
    \item for all $m > 0$, a quiver map $\beta_{m}: \tilde{F}(\Delta^{n})_{m}\rightarrow f^{*}\mathcal{A}_{n-1}$
\end{itemize}
satisfying equations \eqref{equation:MapIntoNdeg} and \eqref{eq:UP}. Note that for all $i\leq j$ in $[n]$, we have $(\Delta^{n})_{m}(i,j) = \Delta_{f}([m],[j-i])$. Any order morphism $f: [m]\rightarrow [n]$ in $\simp$ can be uniquely factored as $[m]\rightarrow [f(m)-f(0)]\hookrightarrow [n]$ where the first map belongs to $\fint$ and the second map is given by $j\mapsto j + f(0)$. It follows that the quiver maps $(\beta_{m})_{m > 0}$ correspond to a collection of elements
$$
a_{f}\in \mathcal{A}_{m-1}(A_{f(0)},A_{f(m)})
$$
for all $m > 0$ and $f: [m]\rightarrow [n]$ in $\simp$. Then equation \eqref{equation:MapIntoNdeg} is equivalent to having, for all $0\leq i\leq m$ and $f: [m]\rightarrow [n]$ in $\simp$:
$$
a_{f\circ \sigma_{i}} =
\begin{cases}
    \id_{A_{i}} & \text{if }m = 0\\
    0 & \text{if }m > 0
\end{cases}
$$
while equation \eqref{eq:UP} is translates to having, for all $m > 0$, $f: [m]\rightarrow [n]$ in $\simp$:
\begin{equation}\label{equation:n-simplexofAinftyNerve}
    \partial(a_{f}) = \sum_{j=1}^{m-1}(-1)^{j-1}a_{f\circ \delta_{j}} - \sum_{\substack{i_{1} + \dots + i_{l} = m\\ i_{j} > 0, l\geq 2}}(-1)^{1+\epsilon(i_1,...,i_l)} m_{l}(a_{f\vert_{[i_{1}]}}\otimes \dots \otimes a_{f\vert_{[i_{l}]}})
\end{equation}
where $f\vert_{[i_{j}]}$ denotes the restiction of $f$ to $[i_{j}]\hookrightarrow [m]: t\mapsto i_{1} + \dots + i_{j-1} + t$. Hence, we conclude that the collection $(a_{f})_{f: [m]\rightarrow [n]}$ is completely determined by those $a_{f}$ for which $f$ is injective. Moreover, notice that equation \eqref{equation:n-simplexofAinftyNerve} is identically zero on both sides for non-injective $f$.

Finally, writing $a_{f} = a_{f(0),\dots,f(m)}$, we conclude that an $n$-simplex of $\tilde{U}(N^{A_{\infty}}_{\mathbb{K}}(\mathcal{A}))$ is equivalent to the data of a choice of objects $A_{0},\dots A_{n}$ of $\mathcal{A}$ along with a collection $(a_{i_{0},\dots,i_{m}})_{0\leq i_{0} < \dots < i_{m} < n}$ of elements $a_{i_{0},\dots,i_{m}}\in \mathcal{A}_{m-1}(A_{i_{0},A_{m}})$ satisfying
\begin{align*}
\partial (a_{i_0,...,i_m})=\sum_{k=1}^{m-1}&(-1)^{k-1} a_{i_0,...,\not i_k,...,i_m}+\\&+\sum_{\substack{2\leq l \leq m\\0 < k_1< \cdots < k_{l-1} < m}} (-1)^{\epsilon_c(k_1,...,k_{l-1})} m_l(a_{i_0,...,i_{k_1}}\otimes \cdots \otimes a_{i_{k_{l-1}},..., i_m})    
\end{align*}
This is precisely the data of an $n$-simplex in the $A_\infty$-nerve \ref{Def:SimplAooNerve}. Thus we have an isomorphism $\tilde{U}(N^{A_\infty}_{\mathbb{K}}(\mathcal{A}))_n\cong N^{A_\infty}(\mathcal{A})_n$. It is quick to see that this is natural in $n$ and $\mathcal{A}$, and thus we have obtained the desiderata.
\end{proof}
\begin{Cor}
\label{Cor:dgNerve}
For any dg-category $\mathcal{C}_{\bullet}$,
    \[
    N^{A_\infty}_\mathbb{K}(\mathcal{C}_{\bullet})\cong N^{dg}_\mathbb{K}(\mathcal{C}_{\bullet})
    \]
\end{Cor}
\begin{proof}
This follows from Theorem \ref{Thm:MapIntoN} and Proposition 3.21 of \cite{lowen2023frobenius}, by observing that $m_l=0$ for $l>2$ if $\mathcal{C}_{\bullet}$ is a dg-category.
\end{proof}
\begin{Thm}
\label{Thm:QCat}
$N^{A_{\infty}}_{\mathbb{K}}(\mathcal{A})$ is a quasi-category in $\mathbb{K}$-vector spaces.
\end{Thm}
\begin{proof}
Set $X = N^{A_{\infty}}_{\mathbb{K}}(\mathcal{A})$ and let $0 < j < n$ be integers and $A,B\in \Ob(\mathcal{A})$. In view of \cite[Corollary 5.3.2]{lowen2024enriched}, let us take $x_{k}\in (X_{k}\otimes_{\Ob(\mathcal{A})} X_{n-k})(A,B)$ and $y_{i}\in X_{n-1}(A,B)$ for all $0 < k,i < n$ with $i\neq j$, such that
\begin{itemize}
\item for all $0 < i < i' < n$ with $i\neq j\neq i'$,
$$
d^{X}_{i'-1}(y_{i}) = d^{X}_{i}(y_{i'}),
$$
\item for all $0 < k < l < n$,
$$
(\id_{X_{k}}\otimes \mu^{X}_{l-k,n-l})(x_{k}) = (\mu^{X}_{k,l-k}\otimes \id_{X_{n-l}})(x_{l})
$$
\item for all $0 < k < n-1$ and $0 < i < n$ with $i\neq j$,
$$
\mu^{X}_{k,n-k-1}(y_{i}) =
\begin{cases}
(d^{X}_{i}\otimes \id_{X_{n-k-1}})(x_{k+1}) & \text{if }i\leq k\\
(\id_{X_{k}}\otimes d^{X}_{i-k})(x_{k}) & \text{if }i > k
\end{cases}
$$
\end{itemize}
Then we must show the existence of an element $z\in X_{n}(A,B)$ such that
\begin{equation}\label{equation:hornfiller}
\mu^{X}_{k,n-k}(z)=x_k\quad (0 < k < n)\quad \text{and}\quad d^{X}_i(z)=y_i\quad  (0 < i < n, i\neq j)
\end{equation}

Let us first define $z$ as an element of $\bigoplus_{g: U\hookrightarrow \Delta^{n}}(s\mathcal{A})_{U}(A,B)$. Given an injective necklace map $g: U\hookrightarrow \Delta^{n}$ such that $g\neq \delta_{j}$ and $g\neq \id_{\Delta^{n}}$, set
$$
z_{g} =
\begin{cases}
    \pi_{g'}(y_{i}) & \text{if }g = \delta_{i}g'\text{ for some }g': U\hookrightarrow \Delta^{n-1}\text{ and }0 < i < n\text{ with }i\neq j\\
    \pi_{g'}(x_{k}) & \text{if }g = \nu_{k,n-k}g'\text{ for some }g': U\hookrightarrow \Delta^{k}\vee \Delta^{n-k}\text{ and }0 < k < n
\end{cases}
$$
It follows from the conditions on $x_{k}$ and $y_{i}$ above that this doesn't depend on the choice of $k$, $i$ or $g'$. This guarantees that the conditions \eqref{equation:hornfiller} will be satisfied. Finally, define:
$$
z_{\id_{\Delta^{n}}} = 0\quad \text{and}\quad z_{\delta_{j}} = \sum_{\substack{i=1\\ i\neq j}}^{n-1}(-1)^{i+j-1}z_{\delta_{i}} + \sum_{\substack{\nu: S\hookrightarrow \Delta^{n}\\ \text{inert}\\ \ell(S)\geq 2}} (-1)^{\epsilon(S)+j-1}m_{\ell(S)}(z_{\nu})
$$
It remains to show that $z$ satisfies the (TAN) relations \eqref{equation:TAN}. By definition, they are satisfied at every $g: U\hookrightarrow \Delta^{n}$ with $g\neq \id_{\Delta^{n}}$ and $g\neq \delta_{j}$ because $y_{i}$ and $x_{k}$ satisfy the (TAN) relations. Further, it is clear from the choice of $z_{\id_{\Delta^{n}}}$ and $z_{\delta_{j}}$ that $z$ satisfies the (TAN) relation at $z = \id_{\Delta^{n}}$. Finally, the (TAN) relation at $g = \delta_{j}$ hold by the following calculation.

First note that by subsequently applying the $A_{\infty}$-relations \ref{Eq:AooRelation}, Lemma \ref{lemma:inertmapscombinatorics}$(b)$, the (TAN) relations for $z$ at every inert $\mu: V\hookrightarrow \Delta^{n}$ with $\mu\neq \id_{\Delta^{n}}$, and Lemma \ref{lemma:inertmapscombinatorics}$(c)$, we have
\begin{align*}
    &\sum_{\substack{\nu: S\hookrightarrow \Delta^{n}\\ \text{inert}\\ \ell(S)\geq 2}} (-1)^{\epsilon(S)} \partial m_{\ell(S)}(z_{\nu})\\
    =& \sum_{\substack{\nu: S\hookrightarrow \Delta^{n}\\ \text{inert}\\ \ell(S)\geq 2}}\sum_{\substack{r + s + t = \ell(S)\\ (r,t)\neq (0,0)}}(-1)^{\epsilon(S)+r+st+1} m_{r + 1 + t}(\id^{\otimes r}\otimes m_{s}\otimes \id^{\otimes t})(z_{\nu})\\
    =& \sum_{\substack{\mu: V\hookrightarrow \Delta^{n}\\ \text{inert}\\ \ell(V)\geq 2}}\sum_{k=1}^{\ell(V)}\sum_{\substack{\nu: W\hookrightarrow \Delta^{n_{k}}\\ \text{inert}}}(-1)^{\epsilon(V)-\ell(V)+\epsilon(W)+\psi_k(\nu)} m_{\ell(V)}(\id^{\otimes k-1}\otimes m_{\ell(W)}\otimes \id^{\otimes l-k})(z_{\mu\circ_{k} \nu})\\
     = &\sum_{\substack{\mu: V\hookrightarrow \Delta^{n}\\ \text{inert}\\ \ell(V)\geq 2}}\sum_{s=1}^{d}(-1)^{\epsilon(V)-\ell(V)+s-1} m_{\ell(V)}(z_{\mu\circ \delta_{i_{s}}}) \\
    = &\sum_{i=1}^{n-1}\sum_{\substack{\mu': V'\hookrightarrow \Delta^{n-1}\\ \text{inert}\\ \ell(V')\geq 2}}(-1)^{\epsilon(V')+i} m_{\ell(V')}(z_{\delta_{i}\circ \mu'})
\end{align*}
where we've written $V = \Delta^{n_{1}}\vee \dots \vee \Delta^{n_{l}}$ and $V^{c} = \{i_{1} < \dots < i_{d}\}$, so that $l = \ell(V)$ and $d = \dim(V)$.

Let us then calculate $\partial(z_{\delta_{j}})$. By applying the definition of $z_{\delta_{j}}$, the previous calculation, and the (TAN) relation for $z$ at $\delta_{i}$ for each $0 < i < n$ with $i\neq j$, we find:
\begin{align*}
    \partial(z_{\delta_{j}}) =& \sum_{\substack{i=1\\ i\neq j}}^{n-1}(-1)^{i+j-1}\partial(z_{\delta_{i}}) + \sum_{\substack{\nu: S\hookrightarrow \Delta^{n}\\ \text{inert}\\ \ell(S)\geq 2}}(-1)^{\epsilon(S)+j-1} \partial m_{\ell(S)}(z_{\nu})\\
    \end{align*}
    \begin{align*}   
    =& \sum_{\substack{i=1\\ i\neq j}}^{n-1}(-1)^{i+j-1}\partial(z_{\delta_{i}}) + \sum_{\substack{i=1\\ i\neq j}}^{n-1}\sum_{\substack{\mu': V'\hookrightarrow \Delta^{n-1}\\ \text{inert}\\ \ell(V')\geq 2}}(-1)^{\epsilon(V')+i+j-1} m_{\ell(V')}(z_{\delta_{i}\circ \mu'}) \\
    &+ \sum_{\substack{\mu': V'\hookrightarrow \Delta^{n}\\ \text{inert}\\ \ell(V')\geq 2}}(-1)^{\epsilon(V')+1} m_{\ell(V')}(z_{\delta_{j}\circ \mu'})\\
    =& \sum_{\substack{i=1\\ i\neq j}}^{n-1}\sum_{l=1}^{n-2}(-1)^{i+j+l} z_{\delta_{i}\circ \delta_{l}} + \sum_{\substack{\mu': V'\hookrightarrow \Delta^{n}\\ \text{inert}\\ \ell(V')\geq 2}}(-1)^{\epsilon(V')+1}m_{\ell(V')}(z_{\delta_{j}\circ \mu'})\\
    =& \sum_{l=1}^{n-2}(-1)^{l-1} z_{\delta_{j}\circ \delta_{l}} + \sum_{\substack{\mu': V'\hookrightarrow \Delta^{n}\\ \text{inert}\\ \ell(V')\geq 2}}(-1)^{\epsilon(V')+1}m_{\ell(V')}(z_{\delta_{j}\circ \mu'})\\
\end{align*}
where in the last equality we have used that $z_{\delta_{i}\circ \delta_{l}} = z_{\delta_{l}\circ \delta_{i-1}}$ for any $0 < l < i < n$. Note that we have recovered exactly the (TAN) relation for $z$ at $\delta_{j}$.
\end{proof}

\printbibliography
\end{document}